\documentclass[11pt,a4paper]{amsart}

\pdfpkmode={lexmarkr}
\pdfmapfile{}

\usepackage{amsmath,amsfonts,mathrsfs}
\usepackage{graphicx}

\newtheorem{theorem}{Theorem}
\newtheorem{corollary}{Corollary}
\newtheorem{proposition}{Proposition}
\newtheorem{lemma}{Lemma}
\theoremstyle{remark}
\newtheorem{assumption}{Assumption}
\theoremstyle{definition}
\newtheorem{definition}{Definition}

\DeclareMathOperator{\dimH}{\mathrm{dim}_\mathrm{H}}
\DeclareMathOperator{\supp}{\mathrm{supp}}

\allowdisplaybreaks[2]

\begin{document}

\title{Shrinking targets in parametrised families}
\date{\today}

\author{Magnus Aspenberg} \address{Centre for Mathematical Sciences,
  Lund University, Box 118, 221 00 Lund, Sweden}
\email{magnusa@maths.lth.se}

\author{Tomas Persson} \address{Centre for Mathematical Sciences, Lund
  University, Box 118, 221 00 Lund, Sweden}
\email{tomasp@maths.lth.se}


\begin{abstract}
  We consider certain parametrised families of piecewise expanding
  maps on the interval, and estimate and sometimes calculate the
  Hausdorff dimension of the set of parameters for which the orbit of
  a fixed point has a certain shrinking target property. This
  generalises several similar results for $\beta$-transformations to
  more general non-linear families. The proofs are based on a result
  by Schnellmann on typicality in parametrised families.
\end{abstract}

\maketitle

\section{Introduction}

Let $T \colon M \to M$ be a dynamical system. In analogy to
Diophantine approximation, Hill and Velani \cite{HillVelani} studied
the set of points $x \in M$ such that the orbit hits a shrinking
target around another point $y$ infinitely often. More precisely, they
studied the set
\[
\{ \, x \in M : d (T^n (x), y) < r_n \text{ i.o.}
\, \},
\]
where $r_n \to 0$ as $n \to \infty$. In their case, $T$ is
an expanding rational map on the Riemann sphere, and $M$ is the
corresponding Julia set. They calculated the Hausdorff dimension of
the set in the case $r_n = e^{-\tau n}$.

This and similar sets have later been studied in several different
settings, for instance for $\beta$-transformations by Bugeaud and Wang
\cite{BugeaudWang} and Bugeaud and Liao \cite{BugeaudLiao}.

One can also study other related sets, such as
\[
\{ \, y \in M : d (T^n (x), y) < r_n \text{ i.o.} \, \},
\]
for fixed $x$, as was done for instance in
\cite{FanSchmelingTroubetzkoy}, \cite{LiaoSeuret} and
\cite{PerssonRams}.

In this paper, we will consider a family of maps $T_a$, where $a$ is a
parameter. For fixed $x$ and $y$, we consider the set of parameters
$a$ for which the orbit of $x$ hits a shrinking target around $y$
infinitely often, that is, we consider sets of the form
\[
\{ \, a : d(T_a^n (x), y) < r_n \text{ i.o.} \, \}.
\]
Sets of this kind were previously studied by Persson and Schmeling
\cite{PerssonSchmeling}, in the case where $T_a$ are
$\beta$-transformations,
\[
T_\beta \colon [0,1] \to [0,1); \quad x \mapsto \beta x \mod 1, \qquad
  (\beta > 1),
\]
and $x = 1$ and $y = 0$. The motivation for studying the set \[
E_\alpha = \{ \, \beta : T_\beta (1) = |T_\beta (1) - 0| \leq
\beta^{-\alpha n} \text{ i.o.}  \,\} \] came from the question how
well the dynamics of $T_\beta$ can be approximated with subshifts of
finite type.

It was shown in \cite{PerssonSchmeling} that the Hausdorff dimension
of $E_\alpha$ is $\frac{1}{1 + \alpha}$. Later, this result was
generalised to arbitrary $y$ in \cite{LiPerssonWangWu}. In
\cite{LiPerssonWangWu}, the point $y$ was also allowed to be a
Lipschitz function of the parameter. There are also some related
results in \cite{LuWu}.

In this paper we will generalise the results of
\cite{PerssonSchmeling} and \cite{LiPerssonWangWu} to families of
maps, more general than the $\beta$-trans\-formations. We will consider
certain parametrised families of piecewise expanding maps of an
interval and prove results analogous to those mentioned above. If $T_a
\colon [0,1] \to [0,1]$, $a \in [a_0,a_1]$, is the family of maps and
$X \colon [a_0,a_1] \to [0,1]$ is a $C^1$ function, then, under some
conditions, we prove that
\[
\frac{1}{1 + \alpha} \leq \dimH \{\, a \in [a_0, a_1] : |T_a^n (X(a))
- y| < e^{-\alpha S_n \log |T_a'| (a)} \text{ i.o.} \,\} \leq s,
\]
where $S_n \log |T_a'| (a) = \sum_{k=1}^n \log |T_a'(T_a^k(X(a)))|$
and $s$ is the root of a pressure function. We show that in some
cases, for instance for $\beta$-trans\-forma\-tions, the lower and upper
bounds coincide, and hence the dimension is $1/(1+\alpha)$.

The proofs in \cite{PerssonSchmeling} and \cite{LiPerssonWangWu} rely
on the symbolic dynamics of the $\beta$-trans\-formations, the so
called $\beta$-shifts. In our more general case, we find it
inconvenient to use this method of proof. Instead, we shall rely on
some results by D.~Schnellmann \cite{Schnellmann} on typical points in
families of piecewise expanding maps on the interval. Schnellmann
studied a point $X(a)$ and its orbit $\{T_a^n (X(a))\}$, and showed
that under some conditions on $X$ and the family $T_a$, for almost all
parameters $a$, the point $X(a)$ is typical with respect to $(T_a,
\mu_a)$, where $\mu_a$ is an invariant measure absolutely continuous
with respect to Lebesgue measure. We say that a point $x$ is typical
with respect to $(T_a, \mu_a)$ if \[ \frac{1}{n} \sum_{k=0}^{n-1}
f(T_a^k (x)) \to \int f \, \mathrm{d}\mu_a, \qquad n \to \infty, \]
whenever $f$ is a continuous function.

The proof in Schnellmann's paper \cite{Schnellmann} uses the method
introduced by M.~Benedicks and L.~Carleson in
\cite{BenedicksCarleson}.  We shall rely on this method, both through
Schnellmann's result and through some large deviation estimates that
we will carry out as a part of our proof.  The method is usually used
to prove that certain properties of a family of dynamical systems hold
for a set of parameters with positive Lebesgue measure. Hence our use
of Benedicks' and Carleson's method is a non-typical use, in the sense
that we use it to investigate properties that only hold for a set of
parameters with zero Lebesgue measure.

\section{Statement of results}

We consider a parametrised family of maps $T_a \colon [0,1] \to
[0,1]$, where $a$ is a parameter that lies in a closed interval $[a_0,
  a_1]$. For every $a \in [a_0,a_1]$, we assume that the map $T_a$ is
uniformly expanding and piecewise $C^{1 + \mathrm{Lip}}$.  (By $C^{1 +
  \mathrm{Lip}}$, we mean functions that are differentiable, with
Lipschitz continuous derivatives.)  It is then well known that there
is a $T_a$-invariant probability measure $\mu_a$ that is absolutely
continuous with respect to Lebesgue measure, see Wong \cite{Wong}. The
entropy of the measure $\mu_a$ can be calculated by the Rokhlin
formula,
\begin{equation} \label{eq:Rokhlin}
  h_{\mu_a} = \int \log |T_a'| \, \mathrm{d} \mu_a,
\end{equation}
see Theorem~3 of \cite{Ledrappier}.

Let $X \colon [a_0,a_1] \to [0,1]$ be a $C^1$ function and fix a point
$y \in [0,1]$. We will investigate the set of parameters $a$ such
that \[ |T_a^n (X(a)) - y| < e^{-\alpha S_n \log |T_a'| (a)} \] holds
for infinitely many $n$. Under some assumptions on the family $T_a$
and on the function $X$, we will prove that the set of such parameters
has Hausdorff dimension which is bounded from above by the root of a
pressure function. We will use that for most parameters, $S_n \log
|T_a'| \approx h_{\mu_a} n$ and this will allow us to prove that the
Hausdorff dimension is bounded from below by $1/(1+ \alpha)$.

We assume that $T_a$ depends on $a$ in a smooth way. More precisely,
our assumptions are as follows.

\begin{assumption} \label{ass:discontinuities}
  There are smooth functions $b_0, \ldots, b_p$ with
  \[
  0 = b_0 (a) < b_1 (a) < \cdots < b_p (a) = 1
  \]
  for every $a \in [a_0,a_1]$, such that the restriction of $T_a$ to
  $(b_i (a), b_{i+1} (a))$ can be extended to a monotone $C^{1 +
    \mathrm{Lip}}$ function on some neighbourhood of $[b_i (a),
    b_{i+1} (a)]$.
\end{assumption}

\begin{assumption} \label{ass:expanding}
  There are numbers $1 < \lambda \leq \Lambda < \infty$ such that
  \[
  \lambda \leq |T_a' (x)| \leq \Lambda
  \]
  holds for all $a \in [a_0,a_1]$ and all $x \in [0,1] \setminus \{b_0
    (a), \ldots, b_p (a)\}$. There is a number $L$ such that $T_a'$ is
    Lipschitz continuous with constant $L$ on each $(b_i(a), b_{i+1}
    (a))$.
\end{assumption}

\begin{assumption} \label{ass:parameterdependence}
  For $x \in [0,1]$, the mappings $a \mapsto T_a(x)$ and $a \mapsto
    T_a' (a)$ are piecewise $C^1$.
\end{assumption}

\begin{assumption} \label{ass:comparablederivatives}
 The $n$th iterate of $X(a)$ as a function of $a$ will be denoted by
 $\xi_n (a)$, that is $\xi_n (a) = T_a^n (X(a))$. We assume that there
 is a constant $c$ and a number $N$ such that
  \begin{equation} \label{eq:derivatives}
    c^{-1} < \biggl| \frac{\xi_n' (a)}{(T_a^n)' (X(a))}
    \biggr| < c, \qquad a \in [a_0, a_1],\ n > N.
  \end{equation}
\end{assumption}

\begin{assumption} \label{ass:density}
  There is a unique invariant measure $\mu_a$ which is absolutely
  continuous with respect to Lebesgue measure. The density is denoted
  by $\phi_a$. On the support of $\mu_a$, the density $\phi_a$ is
  bounded away from zero. We assume that there is a constant $\tau$
  such that
  \begin{equation} \label{eq:densitybound}
    \tau < \phi_a < 1/\tau \text{ on } \supp \mu_a, \qquad a \in
         [a_0,a_1].
  \end{equation}
  Moreover, we assume that there is an open interval $S$ such that $S$
  is contained in the support of $\mu_a$ for any $a \in [a_0,a_1]$.
\end{assumption}

It is now time to define the topological pressure of the family $T_a$.

\begin{definition}
  Given a function (potential) $\phi \colon [0,1] \to \mathbb{R}$, the
  topological pressure is defined as
  \[
  P (\phi, [a_0,a_1]) = \limsup_{n\to \infty} \frac{1}{n} \log
  \sum_{I_n (a)} e^{S_n \phi (a)},
  \]
  where the sum is over the largest open subintervals of $[a_0,a_1]$,
  on which $\xi_n$ is continuous, and $S_n \phi (a) = \sup_{I_n (a)}
  \sum_{k=1}^n \phi \circ \xi_k$.
\end{definition}

Under these assumptions we prove the following theorem.

\begin{theorem} \label{the:main}
  Assume that Assumptions~\ref{ass:discontinuities}--\ref{ass:density}
  hold and let $\alpha > 0$.

  Let $s_0$ be the (unique) root of $s \mapsto P (- s (1 + \alpha)
  \log |T_a'|, [a_0,a_1])$. Then
  \[
  \dimH \{\, a \in [a_0, a_1] : |T_a^n (X(a)) - y| < e^{-\alpha S_n
    \log |T_a'|} \text{ i.o.} \,\} \leq s_0
  \]
  holds for every $y$.

  Suppose $X \colon [a_0, a_1] \to [0,1]$ is such that $X(a)$ is
  typical with respect to $(T_a, \mu_a)$ for almost all $a \in
  [a_0,a_1]$.  If $S$ is an open interval such that $S \subset \supp
  \mu_a$ for every $a \in [a_0,a_1]$, then there is an open and dense
  subset $S_0$ of $S$ such that
  \[
  \frac{1}{1 + \alpha} \leq \dimH \{\, a \in [a_0, a_1] : |T_a^n
  (X(a)) - y| < e^{-\alpha S_n \log |T_a'|} \text{ i.o.} \,\}
  \]
  holds for every $y \in S_0$.
\end{theorem}

The assumed typicality of $X(a)$ for almost all parameters is the main
ingredient in the proof of the lower bound of
Theorem~\ref{the:main}. It will allow us to conclude that for many
large $n$ the values of $T_a^n (x)$ for different parameters $a$ are
well distributed, so that there are plenty of parameters $a$ for which
$|T_a^n (X(a)) - y| < e^{-\alpha S_n \log |T_a'|} \approx e^{-\alpha
  h_{\mu_a} n}$. It is not obvious if $X(a)$ is typical for almost all
$a$, and in fact this need not be the case, for instance when $X(a)$
is a periodic point for all $a$. However, in the next section we
mention some explicit settings, originating from the work of
Schnellmann \cite{Schnellmann}, in which this typicality does
hold. Schnellmann's result is that in several settings, it is possible
to check Assumption~\ref{ass:comparablederivatives} and show that it
implies the almost sure typicality of $X(a)$. We will however need
Assumption~\ref{ass:comparablederivatives} also for other purposes.

In the case that for each $a$, the derivative $|T_a'|$ is constant,
but possibly depending on $a$, we can show that all the bounds in
Theorem~\ref{the:main} coincide, and we get the following corollary.

\begin{corollary} \label{cor:constantderivative}
  Assume that Assumptions~\ref{ass:discontinuities}--\ref{ass:density}
  hold. Let $\alpha > 0$ and suppose $X \colon [a_0, a_1] \to [0,1]$
  is such that $X(a)$ is typical with respect to $(T_a, \mu_a)$ for
  almost all $a \in [a_0,a_1]$. Assume that $x \mapsto |T_a' (a)|$ is
  constant for each $a$.

  If $S$ is an open interval such that $S \subset \supp \mu_a$ for
  every $a \in [a_0,a_1]$, then there is an open and dense subset
  $S_0$ of $S$ such that
  \[
  \dimH \{\, a \in [a_0, a_1] : |T_a^n (X(a)) - y| < e^{-\alpha S_n
    \log |T_a'|} \text{ i.o.} \,\} = \frac{1}{1 + \alpha}
  \]
  holds for every $y \in S_0$.
\end{corollary}

\section{Examples and Corollaries to Theorem~\ref{the:main}}

In this section, we will show some explicit examples for which
Theorem~\ref{the:main} can be applied.

\subsection{Fixed map}

Using Theorem~\ref{the:main}, we can conclude results for a fixed map
as follows (see also \cite{HillVelani}, where a similar result is
proven for a fixed rational map).

\begin{corollary}
  If\/ $T$ is a fixed map which is mixing with respect to an invariant
  measure $\mu$, satisfying the Assumptions~\ref{ass:expanding} and
  \ref{ass:density}, and $\alpha > 0$, then
  \[
  \frac{1}{1 + \alpha} \leq \dimH \{\, x \in [0,1] : |T^n (x) - y|
  \leq e^{- \alpha S_n \log |T'|} \text{ i.o.} \, \} \leq s_0
  \]
  holds for all $y$ in the interior of $\supp \mu$, where
  $s_0$ is the root of the pressure $P(-s (1+\alpha) \log |T'|)$.
\end{corollary}

\begin{proof}
  Take $T_a = T$ for all $a \in [0,1]$ and put $X(a) = a$. Then $X(a)$
  is typical for almost all $a$ according to Birkhoff's Ergodic
  Theorem. Apply Theorem~\ref{the:main} to conclude the result for an
  open and dense set $S_0$ in $\supp \mu$.

  Take any open interval $S \subset \supp \mu$. By Lemma 4.4 in \cite{Liverani}, since we assume that $T$ is mixing,
  for every interval $I$ there is an $N \geq 0$ such that
  \[
  T^N(I) \supset S.
  \]
  Letting $I$ be one of the intervals in $S_0$ we conclude that there
  is an $N \geq 0$ such that for any $y \in S$ there is some $y' \in
  S_0$ such that $T^N(y') = y$. We know that the result holds for
  $y'$.
  For a sequence of numbers $n$ tending to infinity, we have for some
  $z_n$ between $T^n(x)$ and $y'$ that
  \begin{align*}
    |T^{N+n}(x)-T^N(y')| &= |(T^N)'(z_n)||T^n(x)-y'| \\ &\leq
    |(T^N)'(z_n)| e^{-\alpha S_n \log|T'|}.
  \end{align*}
  We note finally that $|(T^N)'(z_n)|$ is bounded by some constant
  (for instance by $\Lambda^N$). It is clear from the proof of
  Theorem~\ref{the:main} that such constants will not change the final
  result. Hence the corollary follows.
\end{proof}

\subsection{$\boldsymbol{\beta}$-transformations and generalised
  $\boldsymbol{\beta}$-transformations}

Suppose $0=t_0 < t_1 < t_2 < \cdots$ are such that $\lim_{n\to\infty}
t_n = \infty$.  Let $T \colon [0,\infty) \mapsto [0,1]$ be a map such
  that for each $n$ the map $T \colon [t_n,t_{n+1}) \mapsto [0,1]$ is
    an increasing $C^2$ map with $T' > 1$.

We will study the family defined by $T_a (x) = T(ax)$, and we call
such families generalised $\beta$-transformations. A simple example is
the usual $\beta$-transformations for which $T (x) = x \mod 1$ and
$T_\beta (x) = T(\beta x) = \beta x \mod 1$, $\beta > 1$.

For families of generalised $\beta$-transformations we have the
following theorem by Schnellmann.

\begin{theorem}[{Schnellmann,
    \cite[Theorem~1.1]{Schnellmann}}] \label{the:schnellmann_beta}
  Suppose $T$ is such that $T(t_n^+) = 0$ for all $n$. If $X \colon
  (1,\infty) \to (0,1]$ is smooth with $X'(a) \geq 0$ for all $a \in
  [a_0, a_1]$ and $T_a(x) = T(ax)$, then the point $X(a)$ is typical
  with respect to $(T_a, \mu_a)$ for almost all $a \in [a_0, a_1]$.
\end{theorem}

We can now conclude the following.

\begin{corollary} \label{cor:beta}
  Suppose $T$ is such that $T(t_n^+) = 0$ for all $n$. Let $T_a (x) =
  T(ax)$, and assume that $X \colon (1,\infty) \to (0,1]$ is such that
  $X' \geq 0$. For any $1 < a_0 < a_1$ and any open interval $S$ with
  $S \subset \supp \mu_a$ for all $a$, if $y \in S$, then we have
  \[
    \frac{1}{1+\alpha} \leq \dimH \{\, a \in [a_0, a_1] : |T_a^n
    (X(a)) - y| < e^{-\alpha S_n \log |T_a'|} \text{ i.o.} \,\} \leq s,
  \]
  where $s$ is the root of the pressure function $P(- s (1+\alpha)
  \log |T_a'|)$.

  In particular, if $|T'|$ is constant, we have
  \[
  \dimH \{\, a \in [a_0, a_1] : |T_a^n (X(a)) - y| < e^{-\alpha S_n
    \log |T_a'|} \text{ i.o.} \,\} = \frac{1}{1+\alpha}.
  \]
\end{corollary}

\begin{proof}
  This follows from Theorem~\ref{the:main} and results by Schnellmann
  \cite{Schnellmann}, including Theorem~\ref{the:schnellmann_beta}.

  For the family $T_a(x) = T(ax)$, it was shown by Schnellmann, that
  there is a unique absolutely continuous invariant measure $\mu_a$,
  \cite[Lemma~5.1]{Schnellmann}, and it was shown that the measure
  $\mu_a$ has a density $\phi_a$ \cite[Section~5.2]{Schnellmann}. On
  the support of $\mu_a$, the density $\phi_a$ is bounded away from
  zero and for any compact interval $I \subset (1,\infty)$ there is a
  constant $\tau$ such that
  \[
    \tau < \phi_a < 1/\tau \text{ on } \supp \mu_a, \qquad a \in [a_0,
      a_1].
  \]
  Moreover, the support of $\mu_a$ is an interval $[0, L(a)]$, where
  $L$ is a piecewise constant function with isolated jumps.

  The estimate \eqref{eq:derivatives} follows by
  \cite[Section~5.1]{Schnellmann}, and $X(a)$ is typical with respect
  to $(T_a,\mu_a)$ for almost all $a$ according to
  Theorem~\ref{the:schnellmann_beta}.

  The remaining assumptions of Theorem~\ref{the:main} are clearly
  satisfied for the family $T_a(x) = T(ax)$. Now the corollary follows
  from Theorem~\ref{the:main}. The fact that we can let $S$ be any
  open interval in$\supp \mu_a$ is explained in
  Section~\ref{sssec:assbeta}.
\end{proof}

The condition \eqref{eq:derivatives} is sometimes difficult to
check. In the case that $X$ is constant, things are simpler, as shown
by the following lemma.

\begin{lemma} \label{lem:Xconst}
  Let $T_a (x) = T(ax)$. Then
  Assumption~\ref{ass:comparablederivatives} is fulfilled for large
  enough $a$, if $X'(a)=0$ and $X > 0$.
\end{lemma}

The proof of Lemma~\ref{lem:Xconst} is in Section~\ref{sec:X}. The
proof lets us also conclude the following result.

\begin{corollary} \label{cor:Xconst}
  We can omit Assumption~\ref{ass:comparablederivatives} in the
  assumptions in Theorem~\ref{the:main}, when $X(a)$ is constant and
  $T(X(a)) > 0$.
\end{corollary}

Suppose now that $X(a)$ and $T_a$ are analytic in $a$. If $X(\tilde
a)$ is pre-periodic or periodic, then we say simply that $X(a)$ is
{\em transversal} near $\tilde a$ if $X(a)$ is not pre-periodic or
periodic for all $a$ in any open neighbourhood of $\tilde a$. Under
these assumption we have the following. (See also \cite{MA-Misse},
\cite{MA} and \cite{BenedicksCarleson} where these methods stems
from.)

\begin{corollary} \label{cor:analytic}
  If\/ $T_a$ is an analytic family and $X(a)$ is analytic and
  transversal in the above sense, then
  Assumption~\ref{ass:comparablederivatives} is satisfied for all $a$
  in some neighbourhood around $\tilde a$.
\end{corollary}

The proof of Corollary~\ref{cor:analytic} is in Section~\ref{sec:X}.

\subsection{Negative $\boldsymbol{\beta}$-transformations}

The negative $\beta$-transformations are maps $T_a (x) = T(ax)$, with
$T(x) = -x \mod 1$ and $a > 1$. These maps were for instance studied
by G\'{o}ra \cite{Gora}, Ito and Sadahiro \cite{ItoSadahiro}, and Liao
and Steiner \cite{LiaoSteiner}.

For any $a > 1$, there is a unique $T_a$ invariant probability measure
which is absolutely continuous with respect to Lebesgue measure,
\cite{Gora,ItoSadahiro}.

\begin{corollary} \label{cor:negbeta}
  Let $[a_0, a_1] \subset (1,\infty)$ be an interval such that
  \[
  2 a (a - [a]) - 2 - a > 0
  \]
  holds for all $a \in [a_0,a_1]$. Suppose $X \colon [a_0,a_1] \to
  \mathbb{R}$ is $C^1$, with $X'(a) > 2 + \frac{a}{a-1}$, and that $y
  \in [0,1]$. Then
  \[
  \dim_H \{\, a \in [a_0,a_1] : |T_a^n (X(a)) - y| < e^{-\alpha \log a n}
  \text{ i.o.} \, \} = \frac{1}{1 + \alpha}.
  \]
\end{corollary}

\begin{proof}
  The proof relies on an extension of the results of Schnellmann by
  Persson~\cite[Corollary 1]{Persson}.

  If $X' (a) > 2 + \frac{a}{a-1}$, then Assumption~2 of \cite{Persson}
  is satisfied. Let
  \[
  \delta = 2 (a - [a]) - 1,
  \]
  and assume that $\delta > 0$. It then holds that, if $I$ is a
  maximal interval of continuity of $T_a$, then
  $(\frac{1-\delta}{2},\frac{1+\delta}{2}) \subset T_a (I)$.

  If moreover $2/\beta < \delta$, then there exists an interval $J$
  of length $1/\beta$, such that $T_a (J) = [0,1)$ and $J \subset
    (\frac{1-\delta}{2},\frac{1+\delta}{2})$.

  The conditions $ \delta = 2 (a - [a]) - 1 > 0$ and $2/\beta <
  \delta$ can be written as $2 a (a - [a]) - 2 - \beta > 0$.

  Now, Corollary~1 of \cite{Persson} proves that $X(a)$ is typical for
  almost all $a$, and that Assumption~\ref{ass:comparablederivatives}
  holds.

  Theorem~\ref{the:main} now finishes the proof: For the parameters we
  are considering, the support of $\mu_a$ is $[0,1]$, see G\'{o}ra
  \cite{Gora}. The reason that we can consider any $y \in [0,1]$ is
  explained further in Section~\ref{sssec:assbeta}.
\end{proof}

For instance, the condition is satisfied for $[a_0,a_1] \subset
(\frac{5+\sqrt{41}}{4},3)$, but it is not satisfied for any value of
$a_0$ smaller than $\frac{5+\sqrt{41}}{4}$.

\subsection{Tent maps}

For $\alpha, \beta > 1$, we define the tent map $T_{\alpha, \beta}
\colon \mathbb{R} \to \mathbb{R}$ by
\[
T_{\alpha, \beta} (x) = \left\{ \begin{array}{ll} 1 + \alpha x &
  \text{if } x \leq 0, \\ 1 - \beta x & \text{if } x > 0.
  \end{array} \right.
\]
Let $\alpha, \beta \colon [a_0,a_1] \to (1, \infty)$ be two
non-decreasing $C^1$-functions such that for all $a \in [a_0,a_1]$
holds
\begin{align*}
  &\alpha' (a) > 0 \text{ for all } a \in [a_0, a_1] \qquad \text{or}
  \qquad \alpha(a_0) = \alpha(a_1),\\ &\beta' (a) > 0 \text{ for all }
  a \in [a_0, a_1] \qquad \text{or} \qquad \beta(a_0) = \beta(a_1)
\end{align*}
and
\[
\frac{1}{\alpha(a)} + \frac{1}{\beta(a)} \geq 1.
\]
We assume also that $(\alpha(a_0),\beta(a_0)) \neq
(\alpha(a_1),\beta(a_1))$, and consider the family $T_a = T_{\alpha
  (a), \beta(a)} \colon [T_{\alpha (a), \beta (a)} (1), 1] \to
[T_{\alpha (a), \beta (a)} (1), 1]$. For each $a$ there is a unique
$T_{\alpha(a), \beta(a)}$-invariant measure $\mu_a$ which is
absolutely continuous with respect to Lebesgue measure. Using results
of Schnellmann \cite[Section~7]{Schnellmann}, we can now state the
following corollary of Theorem~\ref{the:main}.

\begin{corollary}
  Let $\alpha > 0$. Then
  \[
  \frac{1}{1 + \alpha} \leq \dimH \{\, a \in [a_0,a_1] : |T_a^n (0) -
  y| \leq e^{-\alpha S_n \log |T_a'|} \text{ i.o.} \, \} \leq s_0
  \]
  holds for an open and non-empty set of $y$, where $s_0$ is the root
  of the pressure $P(-s (1 + \alpha) \log |T_a'|)$.
\end{corollary}

\begin{proof}
  We consider
  \[
  T_a = T_{\alpha (a), \beta(a)} \colon [T_{\alpha(a),
    \beta(a)} (1), 1] \to [T_{\alpha(a), \beta(a)} (1), 1].
  \]
  That $T_a$ is not defined as a map from $[0,1]$ to itself, is only a
  matter of a coordinate change, and is unimportant. There is a unique
  invariant measure $\mu_a$ that is absolutely continuous with respect
  to Lebesgue, and satisfies all required assumptions, see Schnellmann
  \cite[Section~6--7]{Schnellmann}.

  By Schnellmann \cite[Theorem~7.1]{Schnellmann}, the point $0$ is
  typical for almost all parameters. The corollary follows from
  Theorem~\ref{the:main}.
\end{proof}

\subsection{Families of Markov maps} \label{ssec:markov}

Schnellmann has proved that if $T_a$ is a family of Markov maps, then
$X(a)$ is typical for almost all $a$ \cite[Theorem~8.1]{Schnellmann}.
We say that $T_a$ is a family Markov maps, if for each $a$, the map
$T_a$ is Markov with respect to the partition $0 = b_0 (a) < b_1 (a) <
\cdots < b_p (a) = 1$ mentioned in
Assumption~\ref{ass:discontinuities}. In this case we have the
following result.

\begin{corollary} \label{cor:markov}
  Suppose that $T_a$ is a family of Markov maps satisfying
  Assumptions~\ref{ass:discontinuities}, \ref{ass:expanding} and
  \ref{ass:parameterdependence}, and let $X$ be a $C^1$ map that
  satisfies Assumption~\ref{ass:comparablederivatives}. Then, whenever
  $S$ is an interval with $S \subset \supp \mu_a$ for every $a \in
  [a_0, a_1]$, we have
  \[
  \frac{1}{1 + \alpha} \leq \dimH \{\, a \in [a_0,a_1] : |T^n (X(a)) -
  y| \leq e^{-\alpha S_n \log |T_a'|} \text{ i.o.} \, \} \leq s_0
  \]
  for every $y \in S$, where $s_0$ is the root of the pressure
  function $P (- s (1+\alpha) \log |T_a'|)$.

  In particular, if $|T_a'|$ is constant for each $a$, then
  \[
  \dimH \{\, a \in [a_0,a_1] : |T^n (X(a)) - y| \leq e^{-\alpha S_n
    \log |T_a'|} \text{ i.o.} \, \} = \frac{1}{1 + \alpha}.
  \]
\end{corollary}

\begin{proof}
  As mentioned above, Schnellmann has proved that $X(a)$ is typical
  for almost every $a$. Moreover, Assumption~\ref{ass:density} holds
  in this case.

  We can now almost conclude the corollary, but note that we state the
  result for every $y \in S$, and not only for a dense and open
  subset.

  We have that $\xi_n (I_n(a))$ is uniformly large for every $I_n(a)$,
  which is proved as follows. By, if necessary, considering a smaller
  interval of parameters, we may assume that the endpoints of the
  partition elements, the points $b_i (a)$, are well separated, even
  for different parameters. More precisely, we may assume that there
  is a number $d > 0$ such that for any two parameters $a$ and
  $\tilde{a}$, the points $b_i (a)$ and $b_{i+1} (\tilde{a})$ are at
  least separated by a distance $d$.

  Now, since $T_a$ is a Markov map for each $a$ and since $\xi_n$ is
  piecewise monotone for large $n$ by
  Assumption~\ref{ass:comparablederivatives}, the set $\xi_n (I_n
  (a))$ is an interval of length at least $d$ if $n$ is large enough.
  Hence, there is a lower bound of the lengths of the intervals $\xi_n
  (I_n(a))$.

  So, there is a lower bound on the lengths of the intervals
  $\xi_n (I_n(a))$, which is a property that can be used to get an
  easier proof and a stronger result in Theorem~\ref{the:main}. We
  leave out the details here, and make comments on this in the
  corresponding part of the proof of Theorem~\ref{the:main}, see
  Sections~\ref{sssec:assbeta}.
\end{proof}

\section{Bounded distortion}

In the proof of Theorem~\ref{the:main}, we shall make use of the
following standard bounded distortion estimate. The constant $c$ in
\eqref{eq:derivatives} can be chosen so that for any $a \in [a_0,
  a_1]$, if $I$ is an interval on which $T_a^n$ is continuous
then
\begin{equation} \label{eq:boundeddistortion}
  \biggl| \frac{(T_a^n)' (x)}{(T_a^n)' (y)} \biggr| < c, \qquad x,y
  \in I.
\end{equation}
We shall also need a more general version of bounded distortion to
compare the derivatives for different parameters, see Lemma~4.1 of
\cite{Schnellmann}. A consequence of this lemma is that if $I$ is an
interval on which $\xi_n \colon a \mapsto T_a^n (X(a))$ is continuous,
then
\begin{equation} \label{eq:boundeddistortion2}
  \biggl| \frac{\xi_n' (a)}{\xi_n' (b)} \biggr| < c, \qquad a, b \in
  I.
\end{equation}

\section{Proof of the Upper Bound} \label{sec:upper}

In this section we prove the upper bound of the dimension in
Theorem~\ref{the:main}, as well as
Corollary~\ref{cor:constantderivative}.

Take $s$ such that $P (- s (1 + \alpha) \log |T_a'|) < 0$. We
need to show that the dimension is not larger than $s$.

For any $n > 0$, let $\{I_{n,k}\}$ be the partition of $[a_0,a_1]$
into maximal intervals on which $\xi_n \colon a \mapsto T_a^n (X(a))$
is continuous. For each interval $I_{n,k}$ there is a possibly empty
maximal sub-interval $\hat{I}_{n,k} \subset I_{n,k}$ such that $\xi_n
(\hat{I}_{n,k}) \subset [y - e^{- \alpha S_n \log |T_a'|}, y + e^{-
    \alpha S_n \log |T_a'|}]$. With this notation we have
\[
E = \{\, a \in [a_0, a_1] : |T_a^n (X(a)) - y| < e^{- \alpha S_n \log
  |T_a'|} \text{ i.o.} \,\} \subset \limsup_{n \to \infty} \bigcup_{k}
\hat{I}_{n,k}.
\]

By bounded distortion, the derivative of $\xi_n$ is essentially
constant on $\hat{I}_{n,k}$, and moreover $|\xi_n (\hat{I}_{n,k})|
\leq 2 e^{- \alpha S_n \log |T_a'|}$.  We may therefore estimate that
\begin{equation} \label{eq:lengthofsubinterval}
  |\hat{I}_{n,k}| \leq c' e^{- (1 + \alpha) S_n \log |T_a'|},
\end{equation}
where $c'$ is a constant.

For any $m$, the set $E$ is covered by the sets $\hat{I}_{n,k}$, $n >
m$. Therefore, if we can show that for some $s$,
\[
\sum_{n = 1}^\infty \sum_{k} |\hat{I}_{n,k}|^s < \infty
\]
then it follows that the Hausdorff dimension of $E$ is at most $s$.
We have by \eqref{eq:lengthofsubinterval} that
\[
  \sum_{n = 1}^\infty \sum_{k} |\hat{I}_{n,k}|^s \leq \sum_{n =
    1}^\infty \sum_{k} c'^s e^{- s (1 + \alpha) S_n \log |T_a'|}.
\]
By the definition of the pressure, it now follows that the dimension
of $E$ is at most $s$.

\subsection{Proof of Corollary~\ref{cor:constantderivative}}

Assume that $|T_a'|$ is constant. Then the topological entropy of
$T_a$ is $\log |T_a'|$. We can then conclude that for any $\varepsilon
> 0$, the number of partition elements $I_n (a)$ is bounded by
$e^{(h_+ + \varepsilon) n}$ for large $n$, where $h_+ = \sup_a \log
|T_a'|$.  Let $h_- = \inf_a \log |T_a'|$. Then
\[
\sum_{I_n (a)} e^{- s (1 + \alpha) S_n \log |T_a'|} \leq e^{(h_+ +
  \varepsilon) n} e^{- s (1 + \alpha) h_- n} = e^{(h_+ + \varepsilon -
  s (1 + \alpha) h_-) n}.
\]
This shows that the root of the pressure is at most
\[
\frac{h_+}{(1 + \alpha) h_-}
\]
since $\varepsilon$ is arbitrary. By partitioning the parameter space
$[a_0, a_1]$ into small pieces, we can then conclude that the
Hausdorff dimension is at most $1/(1 + \alpha)$.

\section{Asymptotic behaviour}

In this section, we assume that the
Assumptions~\ref{ass:discontinuities}--\ref{ass:density} hold, and
that $X(a)$ is typical for a.e.\ $a \in [a_0, a_1]$.

Recall that $S$ is such that $S \subset \supp \mu_a$ for every $a \in
[a_0, a_1]$. Whenever $y \in S$, we have by
Assumption~\ref{ass:density} that \[ \inf_{a \in [a_0, a_1]} \mu_{a}
(B(y,l)) \geq \tau l > 0, \] holds for small $l > 0$. This will let us
conclude the following lemma. Let $\lambda$ denote Lebesgue measure.

\begin{lemma} \label{lem:images}
  Let $l > 0$ be small, $\iota > 0$ and suppose that $y \in S$. For
  any subset $\Lambda \subset [a_0, a_1]$ there is an increasing
  sequence of numbers $n_k$ such that
  \[
  \lambda \{\, a \in \Lambda : T_a^{n_k} (X(a)) \in B(y,l) \,\}
  \geq \frac{\tau l}{2} |\Lambda|, \qquad k = 1, 2, 3, \ldots
  \]
  Moreover, the frequency
  \[
  f(n) = \frac{1}{\iota n} \# \{\, k : n \leq n_k \leq (1 + \iota) n
  \,\}
  \]
  has the property that
  \[
  f (n) \geq \frac{\tau l}{4},
  \]
  for sufficiently large $n$.
\end{lemma}

\begin{proof}
  For any $k$, the function
  \[
  a \mapsto \chi_{B(y,l)} (T_a^k (X(a)))
  \]
  is non-negative and bounded by one. Since $X(a)$ is typical with
  respect to $(T_a, \mu_a)$ for almost all $a \in [a_0, a_1]$
  according to the assumption of Theorem~\ref{the:main}, we have for
  almost all $a \in [a_0, a_1]$ that
  \[
  \frac{1}{n} \sum_{k=1}^n \chi_{B(y,l)} (T_a^k (X(a))) \to \mu_a
  (B(y,l)), \qquad n \to \infty,
  \]
  and by the dominated convergence theorem,
  \[
  \int_\Lambda \frac{1}{n} \sum_{k=1}^n \chi_{B(y,l)} (T_a^k
  (X(a))) \, \mathrm{d}a \to \int_\Lambda \mu_a (B(y,l)) \,
  \mathrm{d}a.
  \]

  Since $\mu_a (B(y,l)) \geq \tau l$ we now get that
  \begin{multline*}
    \frac{1}{n} \sum_{k=1}^n \lambda \{\, a \in \Lambda : T_a^k (X(a))
    \in B(y,l) \,\} \\ = \int_\Lambda \frac{1}{n} \sum_{k=1}^n
    \chi_{B(y,l)} (T_a^k (X(a))) \, \mathrm{d}a \to \int_\Lambda
    \mu_a (B(y,l)) \, \mathrm{d}a \geq \tau l |\Lambda|,
  \end{multline*}
  as $n \to \infty$. Hence, the first part of the lemma follows.

  It follows from above that
  \[
  \sum_{k=n}^{(1 + \iota) n} \lambda \{\, a \in \Lambda : T_a^k (X(a))
  \in B(y,l) \,\} \geq \frac{7}{8} \tau l |\Lambda| \iota n
  \]
  if $n$ is large enough.

  When $k$ is such that
  \[
  \lambda \{\, a \in \Lambda : T_a^k (X(a)) \in B(y,l) \,\} \geq
  \frac{\tau l}{2} |\Lambda|,
  \]
  we use the trivial  estimate
  \[
  \lambda \{\, a \in \Lambda : T_a^k (X(a))
  \in B(y,l) \,\} \leq |\Lambda|.
  \]
  Otherwise, we use the estimate
  \[
  \lambda \{\, a \in \Lambda : T_a^k (X(a)) \in B(y,l) \,\} \leq
  \frac{\tau l}{2} |\Lambda|.
  \]
  Hence, by the definition of $f(n)$, we have
  \[
  \sum_{k = n}^{(1 + \iota) n} \lambda \{\, a \in \Lambda : T_a^k
  (X(a)) \in B(y,l) \,\} \leq f(n) \iota n |\Lambda| + (\iota n - f(n)
  \iota n) \frac{\tau l}{2} | \Lambda |.
  \]
  Combining these two estimates, it follows that
  \[
  f (n) \iota n (|\Lambda| - \frac{\tau l}{2} |\Lambda|) \geq \frac{3
    \tau l}{4} |\Lambda| \iota n.
  \]
  This proves the desired estimate if $l$ is sufficiently small.
\end{proof}

We now prove the following lemma, concerning the asymptotic growth of
the derivative of $T_a^n$ and the entropy of the measure $\mu_a$ for
typical $a$.

\begin{lemma} \label{lem:entropy}
  For any subset $\Lambda_0 \subset [a_0, a_1]$ and any $\varepsilon >
  0$, there is a set $\Lambda \subset \Lambda_0$ and a number $N$ such
  that $\lambda(\Lambda_0 \setminus \Lambda) \leq \varepsilon
  \lambda(\Lambda_0)$ and
  \[
  h_{\mu_a} - \varepsilon < \frac{1}{n} \log |(T_a^n)' (X(a))| <
  h_{\mu_a} + \varepsilon, \qquad n > N,\ a \in \Lambda.
  \]
\end{lemma}

\begin{proof}
  This is a consequence of the Rokhlin formula \eqref{eq:Rokhlin} and
  the fact that $X(a)$ is typical for $(T_a,\mu_a)$ for almost all
  $a$.

  If $a$ is such that $X(a)$ is typical, then
  \begin{multline*}
    \frac{1}{n} \log |(T_a^n)' (X(a))| = \frac{1}{n} \sum_{k=0}^{n-1}
    \log |T_a' (T_a^k (X(a)))| \\ \to \int \log |T_a'| \, \mathrm{d}
    \mu_a = h_{\mu_a}, \qquad n \to \infty,
  \end{multline*}
  by Birkhoff's Ergodic Theorem and the Rokhlin formula
  \eqref{eq:Rokhlin}. Since this convergence holds for almost all $a
  \in \Lambda_0$, the lemma follows.
\end{proof}

\section{Large deviations}

Again, throughout this section, we assume that the
Assumptions~\ref{ass:discontinuities}--\ref{ass:density} hold, and
that $X(a)$ is typical for a.e.\ $a \in [a_0, a_1]$.

We outline and develop the so called large deviation argument, first
invented by M.~Benedicks and L.~Carleson in \cite{BenedicksCarleson2},
also used in \cite{MA}.  In order to state the following lemmata we
need a couple of definitions. We recall that we write $I_n (a)$ for
the largest interval around the parameter $a$ such that $\xi_n(a)$ is
continuous on $I_n(a)$. Those intervals are also called {\em partition
  elements}, and we say that $I_n(a)$ is a partition element of {\em
  generation $n$}. In this section we will actually write $I_n (a)$
for the elements of a refined partition. If a partition element is too
long, then we partition it further into pieces of length $\delta \leq
|\xi_n(I_n (a))| \leq 4 c \max |T'_a| \delta$. The number $\delta$
should be chosen such that if $I_n(a)$ is a partition element such
that $\xi_n (I_n (a)) \leq 2 \delta$ then $I_n(a)$ contains at most
$2$ partition elements of generation $n+1$. In particular, $\xi_{n+1}$
is discontinuous on $I_n(a)$ in at most one point.

We say that $n$ is a {\em return time for the parameter} $a$ if
$\xi_{n}$ is not continuous on $I_{n-1}(a)$. In this case,
$\xi_n(I_{n-1}(a))$ has returned, or is a return. We say that $I_n
(a)$ (or $\xi_n(I_n(a))$) is in {\em escape position} if
$\xi_n(I_n(a))$ is of length at least $\delta$.

For the estimates to work below, we need to have $|T_a'| \geq e^5$
(slightly less than $150$). If this is not the case, replace $T_a$
with an iterate of $T_a$ such that this holds.

\begin{definition}[Escape time]
  We consider three cases: If $|\xi_\nu (I_\nu (a))| \geq \delta^2$
  and there exists a $p > 0$ such that $|\xi_{\nu+k}(I_{\nu+k}(a))|
  \geq \delta^2$ for all $0 \leq k < p$ and
  $|\xi_{\nu+p}(I_{\nu+p}(a))| \geq \delta$, then we set
  \[
  E(a,\nu) = 0.
  \]
  If $a$ is such that $|\xi_{\nu+k}(I_{\nu+k}(b))| < \delta$ for all
  $k > 0$, then we set $E(a,\nu)=\infty$. (The set of those parameters
  have measure zero, which can easily be proved using an argument by
  Hofbauer \cite[Lemma~13]{Hofbauer}.)

  In all other cases than the two mentioned above, we set
  \[
  E(a,\nu) = \min \{\, t > 0 : \xi_{\nu+t}(I_{\nu+t}(a)) \text{ is
    in escape position} \, \}.
  \]
  The number $E(a,\nu)$ is called the {\em escape time} of $a$ at
  $\nu$.
\end{definition}

Let $\iota > 0$. For a given parameter $a$, we set
\[
\Theta_n(a) =
\sum_{j=1}^{s(a)} E(a,\nu_j),
\]
where $E(a,\nu_j)$, $j=1,\ldots,s$ are the consecutive escape times
for $a$ in the time interval $[n,(1+\iota)n]$, and $\nu_j \in
[n,(1+\iota)n]$ are return times. With this we mean the following. We
always have $\nu_{j+1} \geq \nu_j + E(a,\nu_j)$, and $\nu_{j+1}$ is
the smallest such $\nu_{j+1}$ such that $\nu_{j+1}$ is a return time
for $a$ and $I_{\nu_{j+1}-1} (a)$ is in escape position.

If $I_n(a)$ is in escape position, very few parameters spend a big
portion of time escaping:

\begin{proposition} \label{largedev}
  If $\delta$ is small enough, then there exists a number $\tau_0 =
  \tau_0(\delta, i) \leq c_0 \iota \delta^{1/3}$, $0 < \tau_0 < 1$,
  such that for $\tau_0 < \tau_1 < 1$, whenever $I_n (a)$ is in escape
  position, we have
  \[
  \lambda (\{\, b \in I_n (a) : \Theta_n (b) \geq \tau_1 n \, \}) \leq
  e^{-(\tau_1-\tau_0)n}|I_n(a)|.
  \]
\end{proposition}

We will also need the following simple lemma.

\begin{lemma} \label{lem:deltaktodelta}
  Given $\delta > 0$ and a number $\delta_k \in (0,\delta)$, there are
  constants $c_1 > 0$ and $K > 0$ such that if $I \subset I_n(a)$ is
  an interval with $|\xi_n (I)| > \delta_k$, then there exists a $0 <
  k < K$ and a $b \in I$ such that $|\xi_{n+k} (I_{n+k} (b))| >
  \delta$ and
  \[
  \frac{|I_{n+k} (b)|}{|I|} > c_1.
  \]
\end{lemma}

\begin{proof}
  As long as $|\xi_{n+k} (I_{n+k} (b))| < 2 \delta$ for all $b \in I_n
  (a)$, there are at most $2^k$ different partition elements of
  generation $n+k$ inside $I$, one of which is at least $2^{-k}$
  times as long as $I_n (a)$, and which we denote by $I_{n+k} (b)$. It
  follows that
  \[
  |\xi_{n+k} (I_{n+k} (b))| \geq c 2^{-k} e^{5k} \delta_k.
  \]
  Hence, within a time $K$, depending only on $\delta$, we will have a
  piece with $|\xi_{n+k} (I_{n+k} (b))| > \delta$, $0 < k < K$, and it
  is clear that the partition element $I_{n+k} (b)$ will make up a
  proportion of $I_n (a)$ that is bounded away from $0$.
\end{proof}

To prove Proposition~\ref{largedev}, we need the following lemma.

\begin{lemma} \label{lem:levelset}
  Suppose that $\nu$ is a return time for $a$ and that
  $|\xi_{\nu}(I_{\nu}(a))| = e^{-r} \leq \delta$. Then for some
  constant $C > 0$, we have
  \[
  \lambda(\{\, b \in I_\nu (a) : E(b,\nu) = t \,\}) \leq C \delta e^{r
    - t(5- \log 2)} |I_\nu (a)|.
  \]
\end{lemma}

\begin{proof}
  Let $t$ be fixed and put
  \[
  A := \{\, b \in I_\nu (a) : |\xi_{\nu+k} (I_{\nu+k} (b))| < \delta,\
  0 < k < t \,\}.
  \]
  We will use the inclusion
  \[
  \{\, b \in I_\nu (a) : E(b,\nu) = t \,\} \subset A
  \]
  and prove that $\lambda (A) \leq C \delta e^{r - t(5- \log 2)}
  |I_\nu (a)|$.

  Take $b \in A$ and consider $I_{\nu+t} (b)$. We recall the bounded
  distortion property (\ref{eq:boundeddistortion}) and the fact that
  the parameter and space derivatives are comparable,
  Assumption~\ref{ass:comparablederivatives}. We have
  \begin{align*}
  \frac{|I_{\nu+t} (b)|}{|I_\nu (a)|} &\leq c \frac{|\xi_\nu
    (I_{\nu+t} (b))|}{|\xi_\nu (I_\nu (a))|} = c e^r |\xi_\nu
  (I_{\nu+t} (b))| \\ & \leq c^2 e^r |\xi_{\nu+t} (I_{\nu+t} (b))|
  e^{-5t} \leq C \delta e^{r-5t}.
  \end{align*}
  For a fixed $t$, we have at most $2^t$ such intervals $I_t (b)$ in
  $A$. (Since at every step $\xi_{\nu+k} (I_{\nu+k} (b))$ is cut into
  at most two pieces, by the definition of $A$ and the choice of
  $\delta$.) Hence the measure of $A$ is at most
  \[
  2^t \max |I_{\nu + t} (b)| \leq 2^t C
  \delta e^{r-5t} |I_\nu (a)| \leq C \delta e^{r - t (5-\log
    2)} |I_\nu (a)|. \qedhere
  \]
\end{proof}

Before we state the next lemma we introduce the notation $A \sim B$,
meaning that there exists a constant $c > 1$ such that
\[
 \frac{1}{c}A \leq B \leq c A,
\]
where $A$ and $B$ are two expressions depending on a number of
variables.

\begin{lemma} \label{growth1}
  Suppose $\omega = I_{\nu+k}(a)$ and $\nu$ is a return time, $k
  >0$. Then for all $b \in \omega$,
  \[
  |\xi_{\nu}(\omega)||(T_b^{k})'(T_b^{\nu}(X(b)))| \sim
  |\xi_{\nu+k}(\omega)|.
  \]
\end{lemma}

\begin{proof}
  First note that by the fact that space and parameter derivatives are
  comparable,
  \[
  \xi_{\nu+k}'(b) \sim (T_b^{\nu+k}(X(b)) = (T_b^k)'(T_b^{\nu}(X(b))
  (T_b^{\nu})'(X(b)).
  \]
  Hence, for $b \in \omega$,
  \begin{align*}
    |\xi_{\nu+k}(\omega)| &\sim |\omega| |\xi_{\nu+k}'(b)| \sim
    |\omega| |(T_b^k)'(T_b^{\nu}(X(b))|| (T_b^{\nu})'(X(b))|
    \\
    &\sim | \omega| |\xi_{\nu}'(b)| |(T_b^k)'(T_b^{\nu}(X(b))| \sim
    |\xi_{\nu}(\omega)||(T_b^k)'(T_b^{\nu}(X(b))|. \qedhere
  \end{align*}
\end{proof}

\begin{lemma} \label{single-escape}
  Suppose $\xi_{\nu-1}(I_{\nu-1} (a))$ is in escape position, and that
  $\nu$ is a return time for $a$. Then we have
  \[
  \int_{I_{\nu-1} (a)} e^{E(b,\nu)} \, \mathrm{d}b \leq |I_{\nu-1} (a)|
  (1+\eta(\delta)),
  \]
  where $\eta(\delta) \leq C_0 \delta^{1/3}$.
\end{lemma}

\begin{proof}
  Note that $I_{\nu-1} (a)$ is split by $\xi_\nu$ into several pieces,
  each on which $\xi_\nu$ is continuous. There can be many pieces,
  who's images under $\xi_\nu$ are at least of length $\delta$. We let
  $\omega_q$, denote the union of these pieces. On $\omega_q$ we have
  $E = 0$.

  Left are at most two pieces, which we call $\omega_{r,1}$ and
  $\omega_{r,2}$. For these pieces, the image $\xi_n (\omega_r)$ could
  be much smaller, where $\omega_r$ is one of $\omega_{r,1}$ and
  $\omega_{r,2}$. It remains to estimate
  \[
  \int_{\omega_r} e^{E(b,\nu)} \, \mathrm{d}b.
  \]
  Let $\omega_r$ be the one of the two pieces $\omega_{r,1}$ and
  $\omega_{r,2}$ for which the above integral is largest.

  Let $|\xi_{\nu}(\omega_r)| = e^{-r}$, where $e^{-r} \leq 4 c \max
  |T_a'| \delta$. The larger piece will escape directly in the next
  step, so therefore we consider only $\omega_r$.

  The set $\omega_r$ is subdivided into two sets:
  \begin{align*}
    \omega_0 &= \{ a \in \omega_r : E(\nu,a) = 0 \} \\
    \omega_1 &= \{ a \in \omega_r : E(\nu,a) > 0 \}.
  \end{align*}
  By definition $\omega_0$ is again subdivided into two sets $\omega_0
  = \omega_0' \cup \omega_0''$ where $\omega_0''$ is the set of
  parameters $b$ that has $\delta > |\xi_{\nu+k}(I_{\nu+k}(b))| \geq
  \delta^2$ for all $k > 0$. Hence $\omega_0'$ is the set of
  parameters which makes escape without becoming smaller than
  $\delta^2$. These are good parameters since $E(b,\nu)=0$ and we now
  turn to $\omega_1$. We subdivide $\omega_1$ into intervals
  $\omega_j$ as follows. Suppose $|\xi_{\nu+k_j}(I_{\nu+k_j}(b))| \leq
  \delta^2$ for the least possible $k_j > 0$.  Then put $\omega_j =
  I_{\nu+k_j}(b)$. Suppose that $|\xi_{\nu+k_j}(I_{\nu+k_j}(b))| =
  e^{-r_j}$. Note that $r_j \geq 2\Delta$, where $\delta =
  e^{-\Delta}$. Set $\omega = I_{\nu-1}(a)$, so that
  $|\xi_{\nu-1}(\omega)| \geq \delta$.

  Now apply Lemma~\ref{growth1} to $\omega_j$ and $k_j+1$, to get ($b
  \in \omega_j$)
  \[
  |\xi_{\nu-1}(\omega_j)||(T_b^{k_j+1})'(b)| \sim
  |\xi_{\nu+k_j}(\omega_j)|.
  \]

  Hence,
  \begin{align*}
    \frac{|\omega_j|}{|\omega|} &\sim \frac{|\xi_{\nu-1}
      (\omega_j)|}{|\xi_{\nu-1} (\omega)|} = \frac{|\xi_{\nu-1}
      (\omega_j)||(T_b^{k_j+1})'
      (\xi_{\nu}(b))|}{|\xi_{\nu-1}(\omega)|
      |(T_b^{k_j+1})' (\xi_{\nu}(b))|} \\
    &\sim \frac{|\xi_{\nu+k_j} (\omega_j)|}{|\xi_{\nu-1} (\omega)|
      |(T_b^{k_j+1})' (\xi_{\nu}(b))|} \leq \frac{e^{-r_j}}{\delta
      |(T_b^{k_j+1})'(\xi_{\nu}(b))| },
  \end{align*}
  where $b \in \omega_j$.

  Now we get, using Lemma~\ref{lem:levelset}
  \begin{align*}
    \int_{I_{\nu-1} (a)} &e^{E(b,\nu)} \, \mathrm{d}b = |I_{\nu-1}
    (a)| + \int_{\omega_{r,1} \cup \omega_{r,2}} e^{E(b,\nu)} \,
    \mathrm{d}b \\ & \leq |I_{\nu-1} (a)| + 2 \int_{\omega_r}
    e^{E(b,\nu)} \, \mathrm{d}b \\ & \leq |I_{\nu-1}(a)| \\ & \qquad +
    2 \sum_j \int\limits_{\{b \in \omega_j: E(b,\nu) < r_j/3 \} }
    e^{E(b,\nu)} \, \mathrm{d}b + \int\limits_{ \{b \in \omega_j:
      E(b,\nu) \geq r_j/3 \} } e^{E(b,\nu)} \, \mathrm{d}b \\ & \leq
    |I_{\nu-1} (a)| + 2 \sum_j \biggl( e^{r_j/3} |\omega_j| +
    \sum_{t=r_j/3}^\infty C \delta e^{r_j - t (5-\log 2)} |\omega_j|
    \biggr) \\ &\leq |I_{\nu-1} (a)| + \sum_j \biggl( C e^{r_j/3}
    \frac{e^{-r_j}}{\delta |(T_a^{k_j})'(\xi_{\nu} (a))|} |I_{\nu-1}
    (a)| \\ & \qquad + C' \delta e^{r_j - r_j(5-\log 2)/3}
    \frac{e^{-r_j}}{\delta |(T_a^{k_j})'(\xi_{\nu} (a))|} |I_{\nu-1}
    (a)| \biggr).
  \end{align*}
  We want to fix the lengths $k_j$ in the above sum. There are at most
  $2^{k}$ small intervals $\omega_j$ that have fixed $k_j=k$. Recall
  that $|(T_a^{k_j})'(\xi_{\nu}(b))| \geq e^{5k_j}$. Summing over the
  lengths $k_j=k$ instead we get
  \begin{align*}
    \sum_j &\biggl( C e^{r_j/3} \frac{e^{-r_j}}{\delta
      |(T_a^{k_j})'(\xi_{\nu} (a))|} |I_{\nu-1} (a)| \\ &\qquad + C'
    \delta e^{r_j - r_j(5-\log 2)/3} \frac{e^{-r_j}}{\delta
      |(T_a^{k_j})'(\xi_{\nu} (a))|} |I_{\nu-1} (a)| \biggr) \\ &\leq
    \sum_{k} \sum_{j, k_j=k} \biggl( C e^{r_j/3}
    \frac{e^{-r_j}}{\delta e^{5k}} |I_{\nu-1} (a)| + C' \delta e^{r_j
      - r_j(5-\log 2)/3} \frac{e^{-r_j}}{\delta e^{5k}} |I_{\nu-1}
    (a)| \biggl) \\ &\leq \sum_{k} \sum_{j, k_j=k} |I_{\nu-1} (a)| ( C
    e^{-2r_j/3 } \delta^{-1}e^{-5k} + C'e^{-r_j(5-\log 2)/3} e^{-5k})
    \\ &\leq \sum_{k} |I_{\nu-1} (a)| ( C e^{-4\Delta/3 } \delta^{-1}
    2^k e^{-5k} + C'e^{-2 \Delta (5-\log 2)/3} 2^k e^{-5k}) \\ &\leq
    C'' \delta^{1/3}.
  \end{align*}
  Finally,
  \[
  \int_{I_{\nu-1} (a)} e^{E(b,\nu)} \, \mathrm{d}b \leq |I_{\nu-1}
  (a)| (1 + C'' \delta^{1/3}),
  \]
  where $\eta(\delta) = C'' \delta^{1/3} \rightarrow 0$ as $\delta
  \rightarrow 0$.
\end{proof}

We want to consider those parameters with exactly $s$ escape
situations in a time interval $[n, (1+\iota) n]$ for some $n > 0$,
i.e.\ there are exactly $s$ return times $\nu$ with $\nu \in
[n,(1+\iota)n]$.

\begin{lemma} \label{escape-int}
  Let $\omega_s$ be the set of parameters in a given partition element
  $I_\nu (a)$, such that every parameter in $\omega_s$ has precisely
  $s$ free returns after escape situations in the time interval
  $[n,(1+\iota)n]$. Then
  \[
  \int_{\omega_s} e^{\Theta_n(b)} \, \mathrm{d}b \leq
  |\omega_s|(1+\eta(\delta))^s \leq |\omega_s| e^{\tau_0 n},
  \]
  for $\tau_0 = \iota \log(1+\eta(\delta)) \leq c_0 \iota
  \delta^{1/3}$.
\end{lemma}

\begin{proof}
  We want to apply Lemma~\ref{single-escape} $s$ times. Note that
  $\omega_s$ is a union of intervals, and for each parameter $a \in
  \omega_s$ there is a nested sequence of intervals $\omega_j' \subset
  \omega_1'$, where $\omega_1'$ is an interval in $\omega_s$,
  $j=1,\ldots,s$ and such that $\omega_{j+1}' \subset \omega_j'$. We
  also include in the definition of $\omega_j'$ that
  $\xi_{\nu_j}(\omega_j')$ is a return after an escape situation,
  writing $\nu_j=\nu_j(a)$ for the return times for the parameter $a$.
  So we have $|\xi_{\nu_j-1}(\omega_j')| \geq \delta$.  Hence
  $\xi_{\nu_s}(\omega_s')$ is the last return after escape situation
  in the time interval $[n,(1+\iota)n]$ for all parameters $a \in
  \omega_s'$.  Since $E(a,\nu_{j-1})$ is constant on $\omega_j'$ we
  get using Lemma~\ref{single-escape} that
  \begin{align*}
    \int_{\omega_{s-1}'}e^{E(b,\nu_{s-1})+E(b,\nu_s)} \, \mathrm{d}b
    &= \sum_{\omega_s' \subset \omega_{s-1}'} e^{E(b,\nu_{s-1})}
    \int_{\omega_s'} e^{E(b,\nu_s)} \, \mathrm{d}b \\ &\leq
    \sum_{\omega_s' \subset \omega_{s-1}'} e^{E(b,\nu_{s-1})}
    (1+\eta(\delta))|\omega_s'| \\ &\leq
    \int_{\omega_{s-1}'}e^{E(b,\nu_{s-1})} \, \mathrm{d}b
    (1+\eta(\delta)) \\ &\leq |\omega_{s-1}'|(1+\eta(\delta))^2.
  \end{align*}
  Repeating this argument $s$ times, it follows that
  \[
  \int_{\omega_1'} e^{\Theta_n(b)} \, \mathrm{d}b \leq
  |\omega_1'|(1+\eta(\delta))^s.
  \]
  Taking the union over all such intervals $\omega_1'$ we get
  \[
  \int_{\omega_s} e^{\Theta_n(b)} \, \mathrm{d}b \leq
  |\omega_s|(1+\eta(\delta))^s.
  \]
  Since $s \leq \iota n$ we may choose $\iota \log(1+\eta(\delta)) =
  \tau_0$ and the lemma follows.
\end{proof}

The proof of Proposition~\ref{largedev} is now short:

\begin{proof}[Proof of Proposition~\ref{largedev}]
  If $\delta$ is small enough, we have $\log(1 + \eta(\delta)) < 1$,
  so that we may choose $\tau_1$ satisfying $\tau_0 = \iota \log(1 +
  \eta(\delta)) < \tau_1 < 1$.

  By Lemma~\ref{escape-int},
  \begin{align*}
    e^{\tau_1 n} \lambda (\{\, b \in I_n (a) : \Theta_n(b) \geq \tau_1
    n \,\}) &\leq \int_{ \{\, b \in I_n (a) : \Theta_n(b) \geq \tau_1
      n \,\} } e^{\Theta_n(b)} \mathrm{d}b \\ &\leq \int_{I_n (a)}
    e^{\Theta_n(b)} \, \mathrm{d}b \leq e^{\tau_0 n}|I_n (a)|,
  \end{align*}
  so
  \[
  \lambda (\{\, b \in I_n (a) : \Theta_n(b) \geq \tau_1 n \,\}) \leq
  e^{ - (\tau_1 - \tau_0)n} |I_n(a)|. \qedhere
  \]
\end{proof}

Since we consider time intervals $[n,(1+\iota)n]$ we can repeat
Proposition~\ref{largedev} on every time interval starting with some
(sufficiently large) $N$ so that the result holds for $[(1+\iota)^j
  N,(1+\iota)^{j+1}N]$ for $j=0, 1, \ldots$, and hence for all $n >
N$. We get the following corollary.

\begin{corollary} \label{cor:largedev}
  If $\delta$ is small enough and $0 < \iota \leq 1$, then there
  exists a number $\tau_0 = \tau_0(\iota, \delta) \leq c_0 \iota
  \delta^{1/3}$, $0 < \tau_0 < 1$, such that for $\tau_0 < \tau_1 <
  1$, whenever $I_n (a)$ is in escape position, we have for $m \geq n$
  that
  \[
  \lambda(\{\, b \in I_n (a) : \Theta_m (b) \geq 3\tau_1 m \,\}) \leq
  2 e^{- (\tau_1-\tau_0)(1 - \iota) m} |I_n(a)|.
  \]
\end{corollary}

\begin{proof}
  Since $I_n (a)$ is in escape position we can apply
  Proposition~\ref{largedev} to the interval $[n, (1 + \iota)
    n]$. Ideally, if for every parameter $b \in I_n (a)$, the
  partition element $I_{(1 + \iota) n}(b)$ is also in escape position,
  we could just apply Proposition~\ref{largedev} over and over again
  and get the result for every interval $[(1 + \iota)^j n, (1 +
    \iota)^{j+1} n]$, where $j \geq 0$. However, it is not quite that
  easy.

  Consider an interval $I_N (\tilde{a})$ in escape position, and the
  corresponding time interval $[N,(1+\iota) N]$. The set of parameters
  that have escape time larger than $\tau_1 N$ inside this interval
  have Lebesgue measure that is an exponentially small fraction of
  $I_{N} (\tilde{a})$, by Proposition \ref{largedev}. More precisely,
  if we put
  \begin{multline*}
    F (\tilde{a}, N) = \{\, b \in I_{N} (\tilde{a}): I_{(1+\iota)N-t}
    (b) \text{ is in escape position} \\ \text{for some $0 \leq t \leq
      \tau_1 N$} \,\},
  \end{multline*}
  then
  \begin{equation} \label{eq:Fcompl}
    \lambda(F(\tilde{a},N)^c) \leq
    e^{-N (\tau_1 - \tau_0)} |I_{N} (\tilde{a})|,
  \end{equation}
  where $F(\tilde{a},N)^c$ stands for the complement of $F
  (\tilde{a},N)$.

  Hence, if we disregard from $F (\tilde{a},N)^c$, we may apply
  Proposition~\ref{largedev} to every partition element
  $I_{(1+\iota)N-t}$ in escape position at time $(1+\iota)N-t$, for $0
  \leq t \leq \tau_1 N$. We thereafter apply
  Proposition~\ref{largedev} to parameters in $F (\tilde{b}, M)$, with
  time $M$ in time intervals of the type $[ N(1+\iota)-t,
  (N(1+\iota)-t)(1+\iota) ]$, and so on.

  Since $t$ depends on the parameter we can follow a parameter $b$ in
  $I_n (a)$ and apply Proposition~\ref{largedev} on a (finite)
  sequence of time intervals $[n_j (b), (1 + \iota) n_j (b)]$ where
  $n_j (b) = (1+\iota) n_{j-1} (b) - t_{j-1} (b)$ and $0 \leq t_{j-1}
  (b) \leq \tau_1 n_{j-1} (b)$. On each new interval we loose
  $e^{-n_j(\tau_1-\tau_0)}|I_{n_j}(b)|$ according to
  \eqref{eq:Fcompl}, which means that the total measure of parameters
  we may have to delete (the corresponding ``bad'' set $F^c$), can be
  made arbitrarily small, if $n=n_1$ is large enough.

  For a fixed parameter $\tilde{b}$, every interval of the type $[m,
    (1+\iota) m]$, where $m \geq n = n_1$, intersects at most two
  intervals of the type $[n_j (\tilde{b}), (1 + \iota) n_j
    (\tilde{b})]$. Suppose that $[m, (1 + \iota) m]$ intersects the
  intervals $[n_j (\tilde{b}), (1 + \iota) n_j (\tilde{b})]$ and
  $[n_{j+1} (\tilde{b}), (1 + \iota) n_{j+1} (\tilde{b})]$. Then $n_j
  \leq m \leq (1 + \iota) m \leq (1 + \iota) n_{j+1}$.

  Consider now a partition element $I_{n_j (b)} (b)$. We write
  \begin{align*}
    I_{n_j (b)} (b) &= F(b,n_j (b)) \cup F(b,n_j(b))^c \\ F(b,n_j(b))
    &= \bigcup_{\tilde{b} \in F(b, n_j(b))} I_{n_{j+1}} (\tilde{b}).
  \end{align*}
  By \eqref{eq:Fcompl}, we have $\lambda(F(b, n_j (b))^c) \leq
  e^{-n_j(b) (\tau_1 - \tau_0)} |I_{n_j (b)} (b)|$ and for each of the
  intervals $I_{n_{j+1}} (\tilde{b})$ with $\tilde{b} \in F (b, n_j
  (b))$, we have by \eqref{eq:Fcompl} that
  \[
  \lambda(F(\tilde{b}, n_{j+1} (\tilde{b} ))^c) \leq e^{-n_{j+1}
    (\tilde{b}) (\tau_1 - \tau_0)} |I_{n_{j+1} (\tilde{b})}| \leq
  e^{-n_j (b) (\tau_1 - \tau_0)} |I_{n_{j+1} (\tilde{b})}|.
  \]
  Summing up, we see that the set of parameters in $I_{n_j} (b)$ with
  escape time at least $\tau_1 n_j + \tau_1 n_{j+1}$ in $[n_j, n_{j+1}
    (1 + \iota)]$ has measure at most
  \[
  2 e^{-n_j (\tau_1 - \tau_0)} |I_{n_{j+1}} (b)|.
  \]
  Taking into account that $n_j \leq m \leq n_{j} (1 + \iota)$, and
  $n_{j+1} \leq n_j (1 + \iota)$ we have
  \[
  \tau_1 n_j + \tau_1 n_{j+1} \leq \tau_1 m + \tau_1 (1 + \iota) m
  \leq \tau_1 m + \tau_1 2 m = 3 \tau_1 m,
  \]
  provided $\iota \leq 1$. Hence,
  \[
  \lambda ( \{\, \tilde{b} \in I_{n_j} (b) : \Theta_m (\tilde{b}) \geq
  3\tau_1 m \,\} ) \leq 2 e^{-n_j (b) (\tau_1-\tau_0)} |I_{n_j} (b)|.
  \]
  Therefore, since $m \leq n_j (b) (1+\iota)$, we have
  \[
    \lambda ( \{\, b \in I_n (a): \Theta_m (b) \geq 3\tau_1 m \,\} )
    \leq 2 e^{-m \frac{\tau_1-\tau_0}{1 + \iota}} |I_n (a)|.
  \]
  This proves the desired estimate, since $(1 + \iota)^{-1} \geq 1 -
  \iota$.
\end{proof}

\section{Proof of the Lower Bound}

To prove Theorem~\ref{the:main}, it remains to estimate the Hausdorff
dimension of the set
\[
E (y) = \{\, a \in [a_0, a_1] : |T_a^n (X(a)) - y| < e^{-\alpha S_n
  \log |T_a'|} \text{ for infinitely many } n \,\}
\]
from below.

Consider any sub-interval $S_1$ of $S$ and let $l = |S_1|$.  Let $0 <
\varepsilon < 1/2$, and $\iota > 0$.  We choose $\tau_0$, $\tau_1$ and
$\delta > 0$ so small that the conclusion of
Corollary~\ref{cor:largedev} holds, and so
that
\begin{equation} \label{eq:sumoffrequencies}
  1 - \frac{3 \tau_1}{\iota} + \frac{\tau l}{2} > 1.
\end{equation}
This is possible since $\tau_1$ can be chosen arbitrary as long as
$\tau_0 < \tau_1 < 1$ and $\tau_0 \leq c_0 \iota \delta^{1/3}$.

Because of \eqref{eq:Rokhlin}, and the fact that $a \mapsto \mu_a$ is
continuous in the weak-* topology, there is an $h$ such
that
\begin{equation} \label{eq:entropy}
  |h-h_{\mu_a}| < \varepsilon \qquad \text{for all } a \in A \subset
  [a_0, a_1],
\end{equation}
provided $A$ is a sufficiently small interval. We take such an
interval $A$ of the form $A = I_n (a)$, such that $I_n (a)$ is in
escape position.

We are going to prove that there is a non-empty open sub-interval $B$
of $S_1$ such that for any $y \in B$ the set
\begin{equation} \label{eq:Hausdorff_epsilon} E_A (y) = \{\, a \in A :
  |T_a^n (X(a)) - y| < e^{-\alpha S_n \log |T_a'|} \text{ for
    infinitely many } n \,\},
\end{equation}
has Hausdorff dimension at least $1 / (1 + \alpha)$. Since $S_1$ is
arbitrary, this implies that $\dimH E(y) \geq 1/(1 + \alpha)$ holds
for a dense and open set of $y \in S$.

Let $Y = \{y_1, y_2, \ldots, y_p\}$ be a set of points in $S_1$ such
that for any $y \in S_1$, there is a $y_k$ with $|y-y_k| <
\delta^2/4$. We will first prove that there is a $y\in Y$ such that
the set $E_A (y)$ satisfies
\[
\dimH E_A (y) \geq \frac{1}{1 + \alpha}
(1 - 17 \varepsilon - \iota).
\]
Since $E_A (y) \subset E (y)$, $Y$ is finite and $\varepsilon > 0$ is
arbitrary, this shows that there exists a $y \in Y$ such that $\dimH
E(y) \geq 1/(1+\alpha)$.  Now, this implies that there is an open and
non-empty interval $B \subset S_1$, such that for each $y \in B$, the
set $E(y)$ has Hausdorff dimension at least $1/(1 + \alpha)$. Indeed,
if this is not the case, then in any sub-interval of $S_1$, however
small, we can find a $y$ such that $\dimH E(y) < 1/(1+\alpha)$. It is
then possible to choose $Y$ such that the sets $E(y_1), \ldots, E
(y_p)$ all have Hausdorff dimension strictly less than $1/(1+\alpha)$,
which would yield a contradiction.

Hence, in order to prove that there exists an open and non-empty
interval $B \subset S_1$, for which $\dimH E(y) \geq 1/(1 + \alpha)$
for all $y \in B$, it suffices to show that
\[
\dimH E_A (y) \geq \frac{1}{1 + \alpha} (1 - 17 \varepsilon - \iota).
\]
holds for at least one $y \in Y$. We will do this below by
constructing a Cantor set $C$ with
\[
C \subset \bigcup_{y \in Y} E_A (y), \qquad \dimH C \geq \frac{1}{1 +
  \alpha} (1 - 17 \varepsilon - \iota).
\]

\subsection{Construction of a Cantor set}

We will define a sequence of families of intervals $\mathscr{I}_k$, $k
\geq 0$, such that
\[
\cup \mathscr{I}_{k+1} \subset \cup \mathscr{I}_k
\]
for all $k$ and
\[
C = \bigcap_{k=1}^\infty \bigcup \mathscr{I}_k \subset \bigcup_{y \in
  Y} E_A (y).
\]

For each $J \in \mathscr{I}_k$, we will define a large integer
$n(J)$. The families $\mathscr{I}_k$ will be constructed to have the
following five additional properties for $k \geq 1$. (Note that we do
not necessarily have these properties for $k = 0$.)

\renewcommand{\theenumi}{\roman{enumi}}
\begin{enumerate}
\item \label{property:length}
  For any $J \in \mathscr{I}_k$ holds
  \[
    e^{-(1 + \alpha) (h + 3\varepsilon) n(J)} \leq |J| \leq
    e^{-(1 + \alpha) (h - 3\varepsilon) n(J)},
  \]
  and
  \[
  (h - 2 \varepsilon)n \leq S_{n(J)} \log |T_a'| \leq (h + 2
  \varepsilon)n, \qquad a \in J.
  \]

\item \label{property:number}
  For any $J \in \mathscr{I}_k$ holds
  \[
    \# \{\, I \in \mathscr{I}_{k+1} : I \subset J \, \} \geq e^{(h - 4
      \varepsilon) n(K)}, \qquad J \supset K \in \mathscr{I}_{k+1}.
  \]

\item \label{property:time} For each $I \in \mathscr{I}_k$, there is a
  number $m$, such that for every $J \in \mathscr{I}_{k+1}$ with $J
  \subset I$, holds
  \[
    m \leq n(J) \leq (1+\iota) m.
  \]

\item \label{property:separation} Let $I \in \mathscr{I}_{k}$ and let
  $m$ be as in \eqref{property:time}. If $J_1$ and $J_2$ are two
  different elements of $\mathscr{I}_k$ that are subsets of $I$, then
  they are separated by at least
  \[
    e^{-(h + 3 \varepsilon) (1 + \iota) m}.
  \]

\item \label{property:image}
  Any $J \in \mathscr{I}_k$ satisfies
  \[
    |\xi_{n(J)} (J)| \geq \delta_k,
  \]
  where $\delta_k > 0$ is a number that only depends on $k$.
\end{enumerate}

Let $\mathscr{I}_0 = \{ A \}$. We define $n (A) = n$.

Suppose that we have constructed $\mathscr{I}_{k-1}$ according to the
properties \eqref{property:length}--\eqref{property:image} above. We
then construct $\mathscr{I}_k$ as follows.

For any $I \in \mathscr{I}_{k-1}$ we will construct certain
sub-intervals of $I$ that will belong to $\mathscr{I}_k$. Let $I \in
\mathscr{I}_{k-1}$ be fixed.

By Lemma~\ref{lem:deltaktodelta}, there is a $c_1 = c_1 (\delta_k)$, a
number $0 < r < K$ and an interval $\tilde{I}$ such that
$|\xi_{n(I)+r} (\tilde{I})| > \delta$ and $|\tilde{I}| \geq c_1 |I|$.

By Lemma~\ref{lem:entropy} and \eqref{eq:entropy}, there is a set $A_1
\subset \tilde{I}$ and a number $N$, such that $\lambda (A_1) \geq (1-
\varepsilon) \lambda (\tilde{I}) \geq (1- \varepsilon) c_1 |I|$
and
\begin{equation} \label{eq:derivativeJk} h - 2 \varepsilon <
  \frac{1}{n} \log |(T_a^n)' (x)| = \frac{1}{n} S_n \log |T_a'| < h +
  2 \varepsilon, \qquad n > N,\ a \in A_1.
\end{equation}

By Lemma~\ref{lem:images}, there is a number $M > N$ and a set $A_2
\subset A_1$ such that $\lambda (A_2) \geq (1 - \varepsilon) c_1 \tau
l |I|$ and that $\xi_n (a) \in S_1$ for a frequency $\tau l / 2$ of $n
> M$.

Take an $m > M$ such that
\[
2 e^{-(\tau_1-\tau_0)(1-\iota) m} < \varepsilon \tau l.
\]
By Corollary~\ref{cor:largedev}, there is a set $A_3 \subset
\tilde{I}$ such that
\[
\lambda (A_3) \geq (1 - \varepsilon \tau l) |\tilde{I}|
\]
and $\Theta_m (b) < 3 \tau_1 m$ for all $b \in A_3$. This implies that
for any $a \in A_3$ the frequency of times $t \in [m, (1+\iota)m]$
such that $|\xi_{t} (I_{t} (a))| \geq \delta^2$ is at least $1 - 3
\tau_1 / \iota$.

Let $A_4 = A_2 \cap A_3$. Then $\lambda (A_4) \geq (1 - 2 \varepsilon)
\tau l c_1 |I|$, and since the sum of the frequencies $\tau l /2$ and
$1 + 3 \tau_1 / \iota$ is larger than 1 by
\eqref{eq:sumoffrequencies}, there is for each $a \in A_4$ an $m (a)
\in [m, (1 + \iota) m]$ such that $\xi_{m(a)} (I_{m(a)} (a)) \in S_1$
and $|\xi_{m (a)} (I_{m (a)} (a))| \geq \delta^2$.  We let $m(a)$ be
the smallest such number in the interval $[m, (1 + \iota) m]$. In this
way we achieve that the partition elements $I_{m(a)} (a)$ and
$I_{m(b)} (b)$, for $a,b \in A_2$, are either disjoint or equal.

Since the set $Y$ is $\delta^2/4$-dense, each set $\xi_{m(a)}
(I_{m(a)} (a))$, $a \in A_2$, hits an element of $Y$. Hence, since $Y$
has $p$ elements, there is at least one $y_j \in Y$ and a
corresponding set $A_5 \subset A_4$
with
\begin{equation} \label{eq:A5measure}
  \lambda(A_5) \geq p^{-1} \lambda (A_4) \geq p^{-1} (1 - 2
  \varepsilon) \tau l c_1 |I|
\end{equation}
and
\[
B (y_j, \delta^2/4) \subset \xi_{m(a)} (I_{m(a)} (a)).
\]

Since for $a \in A_5$, the set $\xi_{m(a)} (I_{m(a)} (a))$ is of
length at least $\delta^2$, and the derivative of $T_a^{m(a)}$
satisfies \eqref{eq:derivativeJk}, we can conclude by
\eqref{eq:derivatives} and \eqref{eq:boundeddistortion2} that
\[
c^{-2} \delta^2 e^{-(h + 2\varepsilon) m(a)} \leq |I_{m(a)} (a)| \leq
c^2 e^{-(h - 2\varepsilon) m(a)}, \qquad a \in A_5.
\]
Hence, by taking $m$ large enough, and using that $m \leq m(a) \leq
(1+\iota) m$, we have
\[
 e^{-(h + 3\varepsilon) (1 + \iota) m} \leq |I_{m(a)} (a)| \leq e^{-(h
   - 3\varepsilon) m}, \qquad a \in A_5.
\]

We have
\[
\bigcup_{a \in A_3} I_{m(a)} \supset A_5,
\]
so, by taking $m$ sufficiently large,
\[
\# \{\, I_{m(a)} (a) : a \in A_5 \,\} e^{-(h - 3\varepsilon) m} \geq
\lambda (A_5) \geq p^{-1} (1 - 2\varepsilon) c_1 \tau l |I| \geq
e^{-\varepsilon m}.
\]
Therefore,
\begin{equation} \label{eq:numberofintervals}
  \# \{\, I_{m(a)} (a) : a \in A_5 \,\} \geq e^{(h - 4\varepsilon) m}.
\end{equation}

We are now ready to choose the intervals in $\mathscr{I}_k$ that are
subsets of $I$. To each $I_{m(a)} (a)$, with $a \in A_5$, there
corresponds an interval $J(a) \subset I_{m(a)} (a)$ with
\[
\xi_{m(a)} (J(a)) = [y_j - e^{-\alpha (h + 2 \varepsilon) m(a)}, y_j +
  e^{-\alpha (h + 2 \varepsilon) m(a)}].
\]
We let the intervals $J(a)$ be the intervals in $\mathscr{I}_k$ that
are subsets of $I$, and put $n(J(a)) = m(a)$.

By the bounded distortion and \eqref{eq:derivativeJk}, the definition
of the intervals $J(a)$ implies that
\[
2 c^{-1} e^{- (1 + \alpha) (h + 4 \varepsilon) m (a)} \leq |J (a)|
\leq 2 c e^{- (1 + \alpha) (h - 4 \varepsilon) m (a)},
\]
so if $m$ is large enough we have
\[
e^{- (1 + \alpha) (h + 5 \varepsilon) m (a)} \leq |J (a)| \leq e^{- (1
  + \alpha) (h - 5 \varepsilon) m (a)},
\]
which proves \eqref{property:length}. Property \eqref{property:number}
holds by \eqref{eq:numberofintervals}.

Since for each $a \in A_5$, the interval $J(a)$ is a subset of
$I_{m(a)} (a)$ that is much smaller than $I_{m(a)} (a)$, we can
conclude that if $J(a)$ and $J(b)$ are two intervals with $a,b \in
A_3$, then they are separated by
\[
e^{-(h + 4 \varepsilon) (1 + \iota) m} \geq e^{-(h + 4 \varepsilon) (1
  + \iota) m(a)},
\]
and $m(a) / m(b) \leq 1 + \iota$. This proves properties (iii) and (iv).

The procedure above is applied to all intervals $I \in
\mathscr{I}_{k-1}$, and in this way we get all intervals of
$\mathscr{I}_k$. By induction, we have constructed $\mathscr{I}_k$
satisfying \eqref{property:length}--\eqref{property:image}.

In conclusion, we have constructed the families of intervals
$\mathscr{I}_k$, such that the Cantor set
\[
C = C(Y) := \bigcap_{k=1}^\infty \bigcup \mathscr{I}_k
\]
has the property that for every $a \in C$, there are two sequences
$n_k$ and $j_k$ such that
\[
|T_a^{n_k} (X(a)) - y_{j_k}| \leq e^{- \alpha (h + \varepsilon) n_k}
\leq e^{-\alpha S_n \log |T_a'|}
\]
holds for every $k$. Clearly, since $Y$ is finite, there is an $y(a)
\in Y$ such that $|T_a^{n_k} (X(a)) - y (a)| < e^{- \alpha S_{n_k}
  \log |T_a'|}$ holds for infinitely many $k$. This shows that for
each $a \in C$ there is a $y \in Y$ such that $a \in E_A (y)$. In
other words, \[ C \subset \bigcup_{y \in Y} E_A (y). \]


\subsection{Estimating the Hausdorff dimension}

It remains to estimate the Hausdorff dimension of the Cantor set $C$,
and show that
\[
\dimH C \geq \frac{1}{1 + \alpha} (1 - 17 \varepsilon - \iota).
\]
This will be done by defining a measure $\mu$ with support in $C$, and
using the mass distribution principle. The measure $\mu$ is defined as
follows. For any $I \in \mathscr{I}_0$, we define $\mu (I) =
1$. Suppose that $\mu (I)$ has been defined for all $I \in
\mathscr{I}_{k-1}$.  Then, if $J \in \mathscr{I}_k$ and $J \subset I
\in \mathscr{I}_{k-1}$, we define $\mu (J)$ as follows. By
construction, the interval $J$ is contained in the partition element
$I_{n (J)} (a)$, $a \in J$. We let $F_k$ be the union of all such
partition elements,
\[
F_k = \bigcup_{a \in K \in \mathscr{I}_k} I_{n (K)} (a).
\]
We then define $\mu (J)$ by
\begin{equation} \label{eq:mudefinition}
  \mu (J) = \frac{\lambda(I_{n (J)} (a)))}{\lambda (I \cap F_k)} \mu
  (I), \qquad a \in J.
\end{equation}
By induction this defines $\mu(I)$ for any $I \in \mathscr{I}_k$, $k
\geq 0$, and $\mu$ can be uniquely extended to a Borel probability
measure on $[a_0, a_1]$.

Suppose that $I \in \mathscr{I}_k$ for some $k >
0$. Let
\begin{equation} \label{eq:choiseofs}
  0 < s < \frac{1}{1 + \alpha} (1 - 17 \varepsilon - \iota) <
  \frac{1}{1 + \alpha} \biggl( \frac{h - 13 \varepsilon}{h + 4
    \varepsilon} - \iota \biggr).
\end{equation}
We first show that
\begin{equation} \label{eq:measureestimate1}
 \mu (I) \leq |I|^s.
\end{equation}
After we have done so, we will use this estimate to show a similar
estimate for a general interval $I$.

Let $a \in J \in \mathscr{I}_{k-1}$ such that $I \subset J$. By
\eqref{eq:mudefinition} and the properties
\eqref{property:length}--\eqref{property:image}, we
have
\begin{align*}
  \mu (I) &= \frac{\lambda(I_{n(I)} (a))}{\lambda(J \cap F_k)} \mu (J)
  \leq \frac{e^{- (h - 4 \varepsilon) n (I)}}{e^{- (h + 4 \varepsilon)
      (1 + \iota) n (I)} e^{(h - 5 \varepsilon) n (I)}} \mu (J) \\ &=
  e^{- (h - 13 \varepsilon - \iota (h + 4 \varepsilon)) n (I)} \mu (J)
  \\ &\leq e^{- (h - 13 \varepsilon - \iota (h + 4 \varepsilon)) n
    (I)} \leq |I|^s,
\end{align*}
where we used \eqref{eq:choiseofs} and $e^{-(1+\alpha)(h + 4
  \varepsilon) n (I)} \leq |I|$ in the last step.

Let $I \subset [a_0, a_1]$ be an interval, and suppose that $\mu (I) >
0$. Then there is a smallest number $l \geq 0$ with the property that
there exists an element of $\mathscr{I}_l$ that is a subset of
$I$. Let $J$ be an element of $\mathscr{I}_l$ with $J \subset I$.

For any $k < l$, there is at most two and at least one element of
$\mathscr{I}_k$ with non-empty intersection with $I$. We suppose that
there are exactly two intervals $J_1$ and $J_2$ in $\mathscr{I}_{l-1}$
with non-empty intersection with $I$. (The case with only one such
interval is simpler and can be treated in a similar way.)

We partition $I$ into two parts $I_1$ and $I_2$, corresponding to
$J_1$ and $J_2$, that is we partition $I$ so that
\begin{align*}
  &I = I_1 \cup I_2,&& I_1 \cap I_2 = \emptyset,\\ &I_1 \cap J_2 =
  \emptyset, && I_2 \cap J_1 = \emptyset.
\end{align*}

We may assume that $J \subset J_1$ and hence $J \subset I_1$. By
\eqref{property:length}, we then have the
estimates
\begin{equation} \label{eq:lengthofI1}
  |I| \geq |I_1| \geq |J| \geq e^{-(1 + \alpha)(h + 4 \varepsilon) n
    (J)}.
\end{equation}

Either $\mu (I_2) = 0$ or $\mu (I_2) > 0$. In the later case, there
exists a smallest number $q$ such that there exists an interval
$\hat{J} \in \mathscr{I}_q$ with $\hat{J} \subset I_2$. We then have
the estimates
\[
|I| \geq |I_2| \geq |\hat{J}| \geq e^{-(1 + \alpha)(h + 4 \varepsilon)
  n (\hat{J})}.
\]

Since $\mu (I) = \mu (I_1) + \mu (I_2)$ we have
\[
\frac{1}{2} \mu (I) \leq \max \{ \mu (I_1), \mu (I_2) \}.
\]
We will show that there is a constant $c_2$, independent of $I$, $I_1$
and $I_2$ such that
\[
\mu (I_1) \leq c_2 |I_1|^s, \qquad \mu (I_2) \leq c_2 |I_2|^s.
\]
It then follows that
\[
\mu (I) \leq 2c_2 \max \{ |I_1|^s, |I_2|^s \} \leq 2c_2 |I|^s.
\]
By the mass distribution principle, this shows that
\[
\dimH C \geq s.
\]

We show that $\mu (I_1) \leq c_2 |I_1|^s$. The corresponding
inequality for $I_2$ is proved in a very similar way.

Take $m$ such that $m \leq n (\omega) \leq (1+ \iota)m$ for all
$\omega \in \mathscr{I}_l$ that are subsets of $I_1$. This is possible
by \eqref{property:time}.

Suppose first that $|I_1| \geq e^{-(h - 4 \varepsilon) m}$. We let
$\tilde{I}_1$ be the shortest interval $\tilde{I}_1 \supset I_1$ such
that if $a \in K \in \mathscr{I}_l$, then $\tilde{I}_1$ contains
$I_{n(K)} (a)$. Hence $\tilde{I}_1$ can be obtained by slightly
expanding the interval $I_1$, and since any interval $I_{n(K)} (a)$ is
at most as long as $I_1$, we have $|I_1| \leq |\tilde{I}_1| \leq 3
|I_1|$.

By the definition of $\mu$ we have
\[
\mu (I_1) \leq \frac{\lambda(\tilde{I}_1 \cap F_l)}{\lambda (J_1 \cap
  F_l)} \mu (J_1).
\]
We now have by \eqref{eq:A5measure} and \eqref{eq:measureestimate1}
that
\begin{multline*}
  \mu (I_1) \leq \frac{|\tilde{I}_1|}{p^{-1} (1 - 2 \varepsilon) \tau
    l c_1 |J_1|} |J_1|^s \\ \leq \frac{1}{p^{-1} (1 - 2 \varepsilon)
    \tau l c_1} |\tilde{I}_1|^s \biggl( \frac{|\tilde{I}_1|}{|J_1|}
  \biggr)^{1 - s} \leq c_3 |I_1|^s.
\end{multline*}

In the case that $e^{-(h + 4 \varepsilon) (1 + \iota) m} \leq |I_1| <
e^{- (h - 4 \varepsilon) m}$, we have that $I_1$ intersects at most
\[
\frac{|I_1|}{e^{-(h+4\varepsilon) (1+\iota)m}} + 2 \leq 3
\frac{|I_1|}{e^{-(h+4\varepsilon) (1+\iota)m}}
\]
intervals from $\mathscr{I}_l$. We then have that
\begin{align*}
  \mu (I_1) &\leq e^{-(h - 4 \varepsilon) m} \frac{3
    |I_1|}{e^{-(h+4\varepsilon) (1+\iota)m}} \frac{1}{\lambda(J_1 \cap
    F_l)} \\ &\leq 3 |I_1|^s e^{-(h - 4 \varepsilon) m} \frac{e^{- (1
      - s)(h - 4 \varepsilon)m}}{e^{-(h + 4 \varepsilon) (1 + \iota)
      m}} \frac{1}{\lambda(J_1 \cap F_l)} \\ &\leq 3 |I_1|^s e^{-(h - 4
    \varepsilon) m} \frac{e^{- (1 - s)(h - 4 \varepsilon)m}}{e^{-(h +
      4 \varepsilon) (1 + \iota) m}} \frac{1}{e^{-(h + 4
      \varepsilon)(1 + \iota)m} e^{(h - 5 \varepsilon) m}} \\ &\leq 3
  |I_1|^s,
\end{align*}
if $\varepsilon$ and $\iota$ are small enough.

Otherwise, we have by \eqref{eq:lengthofI1} that $e^{-(1 + \alpha)(h +
  4 \varepsilon) m} \leq |I_1| < e^{- (h + 4 \varepsilon) (1 + \iota)
  m}$ and $I_1$ intersects at most $2$ intervals from $\mathscr{I}_l$.
Then
\[
\mu (I_1) \leq \frac{2 e^{- (h - 4 \varepsilon) m}}{e^{- (h + 4
    \varepsilon)(1 + \iota)m} e^{(h - 5 \varepsilon) m}} \mu (J_1) \leq
2 e^{-(h + 13 \varepsilon + \iota (h + 4\varepsilon)) m} \leq 2
|I_1|^s.
\]

Hence, in all cases we have $\mu (I) \leq c_2 |I|^s$, where $c_2 =
\max \{3, c_3\}$. This proves that $\dimH C \geq s$ and hence finishes
the proof.

\subsection{Comments on families of $\boldsymbol{\beta}$-transformations and
  Markov maps} \label{sssec:assbeta}

Suppose that we are in the setting of Section~\ref{ssec:markov}, that
is, $T_a$ is a family of Markov maps.  In this case, we can avoid
using the large deviation estimate found in
Corollary~\ref{cor:largedev}. This is because we have that $\xi_n (I_n
(a))$ is large, and one can use this to replace the use of
Corollary~\ref{cor:largedev}. The effect is that it is then possible
to construct the Cantor set without first choosing the finite set $Y$,
and in fact one can construct directly a Cantor set inside $E(y)$ for
any desired $y \in S$.

If the family $T_a$ is a family of generalised
$\beta$-transformations, or negative $\beta$-transformations, then the
same phenomenon takes place. For instance, for
$\beta$-transformations, most of the images $\xi_n (I_n(a))$ will be
intervals of the form $[0, b)$.

This comment explains why we get the lower bound for a potentially
larger set of $y$, in Corollaries~\ref{cor:beta}, \ref{cor:negbeta}
and \ref{cor:markov}.

\section{Proofs of Lemma~\ref{lem:Xconst} and Corollary~\ref{cor:analytic}} \label{sec:X}

\begin{proof}[Proof of Lemma~\ref{lem:Xconst}]
  In this proof, write $x=X$ which is now constant and does not depend
  on $a$. We consider the second iterate of $x$, namely $T_a^2(x) =
  T_a(T_a(x))$. Writing $T_a(x)= T(ax)$ (recall $\xi_n(a) = T_a^n(x)$)
  one readily verifies
  \[
  \xi_2'(a) = \frac{\partial }{\partial a} T(aT(xa)) =
  T'(aT(ax))(T(ax) + axT'(ax)).
  \]
  Since both $T(ax)$ and $T'(ax)$ are positive, and moreover $a > 1$
  and $x > 0$ we can easily choose $a$ to fulfill $|\xi_2'(a)| \geq
  (\min(T_a'(x)))^2$. Let $\gamma = \min \log |T_a'(x)|$. Then
  $\xi_2'(a) \geq e^{2 \gamma}$.  We now proceed by induction, to
  prove that the required assumption is satisfied (in fact, we mimic
  the proof of Proposition~4.6 in \cite{MA}). As induction assumption,
  suppose that for some $k\geq 0$, we have
  \[
  \xi_{k+2}'(a) \geq e^{(k+2) \gamma'},
  \]
  where $\gamma' = \gamma/2$. From the above, this is true when
  $k=0$. We prove that
  \begin{equation} \label{xigrowth}
    \xi_{k+3}'(a) \geq e^{(k+3) \gamma'}.
  \end{equation}
  Writing $T(a,x) = T(ax)$, we have the following recursion formula,
  \begin{equation} \label{da}
    \frac{\partial T^{n+1}(x,a)}{\partial a} = \frac{\partial
      T(\xi_n(a),a)}{\partial x} \frac{\partial T^n(x,a)}{\partial a}
    + \frac{\partial T(\xi_n(a),a)}{\partial a}.
  \end{equation}
  We have the explicit derivatives $T_a'(x)=aT(ax)$ and
  $\frac{\partial}{\partial a} T_a(x) = xT(ax)$. From this and the
  recursion formula (\ref{da}), together with the induction assumption
  give
  \begin{align}
    |\xi_{2+k+1}'(a) | &\geq |T_a'(\xi_{2+k}(a))| |\xi_{2+k}'(a)
    |\biggl(1- \frac{| \partial_a T_a(\xi_{2+k}(a)) | }
    {|T_a'(\xi_{2+k}(a))| |\xi_{2+k}'(a)|} \biggr) \label{xiprim} \\
    &\geq |(T_a^{k+1})'(\xi_2(a))||\xi_2'(a)| \prod_{j=0}^k
    \biggl(1-\frac{|\partial_a T_a(\xi_{2+j}(a))
      |}{|T_a'(\xi_{2+j}(a))||\xi_{2+j}'(a)|} \biggr)
    \nonumber \\
    &\geq e^{\gamma (k+1)} e^{\gamma' 2} \prod_{j=0}^k (1-\frac{x}{a}
    e^{-\gamma'(2+j)} ) .  \nonumber
  \end{align}
  Since $a$ can be chosen arbitrarily much larger than $x$ it is clear
  that the sum
  \begin{equation} \label{sum}
    \sum_{j=0}^{\infty} \frac{x}{a} e^{-\gamma'(2+j)}
  \end{equation}
  can be chosen arbitrarily small. Therefore, we may achieve,
  \begin{equation*}
    |\xi_{2+k+1}'(a) | \geq e^{(\gamma - \gamma')(k+1)}
    e^{\gamma'(2+k+1)} \prod_{j=0}^k (1-\frac{x}{a} e^{-\gamma'(2+j)}
    ) \geq e^{\gamma'(2+k+1)}.
  \end{equation*}
  Since the induction assumption for $k=0$ is true, (\ref{xigrowth})
  follows for all $k\geq 0$.

  Now, put
  \[
  Q_n(a) = \frac{(T_a^n)'(x)}{\xi_n'(a)}.
  \]
  Similar to (\ref{xiprim}) but without absolute values we get
  \begin{align*}
    \xi_{2+k+1}'(a) &= T_a'(\xi_{2+k}(a)) \xi_{2+k}'(a) \biggl(1 +
    \frac{ \partial_a T_a(\xi_{2+k}(a))} {T_a'(\xi_{2+k}(a))
      \xi_{2+k}'(a)} \biggr) \label{xiprim2} \\ &=
    (T_a^{k+1})'(\xi_2(a)) \xi_2'(a) \prod_{j=0}^k
    \biggl(1+\frac{\partial_a T_a(\xi_{2+j}(a)) }{T_a'(\xi_{2+j}(a))
      \xi_{2+j}'(a) } \biggr)
  \end{align*}
  Hence
  \begin{equation} \label{QN}
  Q_{2+k+1}(a) = \frac{\xi_{2+k+1}'(a)}{(T_a^{2+k+1})'(x)} = Q_2(a)
  \prod_{j=0}^k \biggl(1+\frac{\partial_a T_a(\xi_{2+j}(a))
  }{T_a'(\xi_{2+j}(a)) \xi_{2+j}'(a) } \biggr).
  \end{equation}
  Since the sum (\ref{sum}) can be made arbitrarily small, we may
  easily achieve
  \[
  |Q_{2+k}(a) - Q_2(a)| \leq |Q_2(a)|/100,
  \]
  (for instance). Hence the space and parameter derivatives are
  comparable and (\ref{eq:derivatives}) is fulfilled.
\end{proof}

Note that the lemma can also be concluded from \eqref{xigrowth} using
a result of Schnellmann \cite[Lemma~2.1]{Schnellmann}.

As a matter of fact, the $a$ dependence of the starting point $x$ is
illusory when dealing with Hausdorff dimension, since we can consider
the function $T(Kax)$ for some large $K$, instead of increasing
$a$. We then have \[ \xi_2'(a) = \frac{\partial }{\partial a}
T(KaT(Kax)) = T'(Ka T(Kax))(KT(Kax) + KaxT'(Kax)). \] If $T(Kax) > 0$
we can then conclude the same result as above, by taking $K$
sufficiently large. Since Hausdorff dimension does not change under
linear maps we conclude Corollary~\ref{cor:Xconst}.

\begin{proof}[Proof of Corollary~\ref{cor:analytic}]
  Suppose for simplicity that $X(\tilde a)$ belongs to a periodic
  point $p(\tilde a)$ of period $q=1$, i.e.\ a fixed point and that
  $\tilde{a}=0$. Of course the fixed point has to be repelling. If
  $T_a$ is an analytic family of maps then $p(a)$ is analytic by the
  Implicit function theorem. Also the multiplier $\lambda_a =
  T_a'(p(a))$ is analytic. We prove first that $\xi_n'(a)$ grows
  exponentially for $a$ close to $\tilde{a}$ where $n$ is as large as
  possible but such that $\xi_k(a)$ belongs to some neighbourhood of
  $p(a)$ for all $k \leq n$.  Put
  \[
  h(a) = X(a) - p(a),
  \]
  for $a$ close to $\tilde{a}$. By the transversality condition, $h$
  is not identically equal to zero. We have that $h(0) = 0$ and hence
  \begin{equation} \label{funch}
  h(a) = Ka^m + \ldots,
  \end{equation}
  where $K \neq 0$ and $m >0$. Recall that $\xi_k (a) = T_a^k
  (X(a))$. We now define an ``error function'' $E_k (a)$ via the
  equation
  \begin{equation} \label{xin}
    \xi_k (a) = h(a)\lambda_a^k  + p(a) + E_k (a).
  \end{equation}
  Arguing the same way as in
  \cite{MA} we differentiate (\ref{xin}) and obtain
  \begin{equation} \label{xinprim}
    \xi_k' (a) = \lambda_a^k \biggl( h'(a) + k h(a)
    \frac{\lambda_a'}{\lambda_a} + \frac{p'(a) + E_k'
      (a)}{\lambda_a^k} \biggr).
  \end{equation}
  The main point now is that $h'(a)$ is the dominant term above. Since
  $T_a (X(a))$ is close to $p(a)$ for small $a$ (remember $h(0) = 0$)
  we can use that the map $T_a$ is conjugate to the linear map $l_a(x)
  = \lambda_a (x-p(a)) + p(a)$ in a neighbourhood of $p(a)$. In other
  words, there exists a real-analytic map $\phi_a$, mapping a
  neighbourhood of $p(a)$ to itself, such that
  \[
    \phi_a \circ T_a(z) = l_a(x) \circ \phi_a(z),
  \]
  for $z$ in a neighbourhood of $p(a)$. Moreover, $\phi_a$ and its
  inverse satisfy
  \begin{align*}
    \phi_a (z) &= p(a) + z-p(a) + \mathcal{O} ((z-p(a))^2)
    \\ \phi_a^{-1}(z) &= p(a) + z-p(a) + \mathcal{O} ((z-p(a))^2),
  \end{align*}
  see e.g.\ \cite[page 31--33]{CarlesonGamelin}.

  The conjugation function $\phi_a$ is valid in a small neighbourhood
  $\mathcal{N}(a)$ of $p(a)$, for $a \in U$ where $U$ is a
  sufficiently small interval around $\tilde{a}$. We pick a
  neighbourhood $\mathcal{N}$ so that $\mathcal{N} \subset
  \mathcal{N}(a)$ for all $a \in U$. Let us choose some $0 < \eta < 1
  $ such that
  \[
  \{x : |x-p(a)| \leq \eta \} \subset \mathcal{N},
  \]
  for all $a \in U$. Then we choose $n=n(a)$ as large as possible such
  that
  \[
  |T_a^n(X(a)) - p(a)| = |\phi_a^{-1}(l_a^n \circ \phi_a(X(a))) -
  p(a)| \leq \eta.
  \]
  Since $\phi_a$ and $\phi_a^{-1}$ are real-analytic, this implies
  that $|l_a^n \circ \phi_a(X(a)) - p(a)| \leq C \eta$ and
  $|X(a)-p(a)| |\lambda_a^n| \leq C\eta$ for some constant $C$. For a
  fixed $k$, we get
  \begin{align}
    T_a^k (X(a)) &= \phi_a^{-1} (l_a^k \circ \phi_a(X(a))) =
    \phi_a^{-1} (p(a) + \lambda_a^k (\phi_a(X(a)) - p(a)) ) \nonumber
    \\ &= p(a) + \lambda_a^k (\phi_a(X(a)) - p(a)) + \mathcal{O}(
    (\lambda_a^k (\phi_a(X(a)) - p(a)))^2 ) \nonumber \\ &= p(a) +
    \lambda_a^k (X(a) - p(a) + \mathcal{O}( (X(a) - p(a))^2 ) ) \nonumber
    \\ & \hspace{1cm} + \mathcal{O} ( (\lambda_a^k (X(a) - p(a) +
    \mathcal{O}( (X(a) - p(a))^2 ) ) )^2 ) \nonumber \\ &= p(a) +
    \lambda_a^k (h(a) + \mathcal{O} (h(a)^2) ) + \mathcal{O} (
    (\lambda_a^k h(a))^2 ), \label{ordo}
  \end{align}
  as $a \to \tilde{a}$. We also have
  \begin{equation} \label{eq:En}
    \xi_{n(a)} (a) = T_a^{n (a)} (X(a)) = p(a) + \lambda_a^{n (a)}
    h(a) + E_{n(a)} (a).
  \end{equation}
  By assumption $|h(a)| |\lambda_a^{n(a)}| = |X(a) - p(a)|
  |\lambda_a^{n(a)}| \leq C \eta$, so if $\eta$ is small enough we can
  compare equation~\eqref{ordo} and \eqref{eq:En}, which shows that
  $|E_{n(a)} (a)|$ is much smaller than $|\lambda_a^{n(a)} h(a)| +
  |p(a)|$, provided $a \in U$ and $U$ is a sufficiently small interval
  around $\tilde{a}$.  This shows that $E_{n(a)} (a)$ is uniformly
  bounded in $a$ for some small parameter interval $U$.

  By definition of $n(a)$, we have that $n (a) + 1$ is the smallest
  integer for which the inequality $|T_a^{n+1}(X(a))-p(a)| \leq \eta$
  is not true. For $a \in U$ and $n=n(a)$ we then have
  \[
  \eta \geq |\xi_n(a) - p(a)| = | \lambda_a^n h(a) + E_n(a) | \geq
  \eta/(2\Lambda),
  \]
  where $\Lambda$ is defined in Assumption~\ref{ass:expanding}. Since
  $|E_{n(a)} (a)|$ is much smaller than $|\lambda_a^{n(a)} h(a)|$,
  $\lambda_a^{n(a)} h(a)$ dominates and thus
  \[
  n(a) \log |\lambda_a| + \log |h(a)| \sim \log \eta.
  \]
  Hence for small $a$ (close to $\tilde{a}=0$) we have
  \[
  n (a) \sim - \log |h(a)|.
  \]
  Since $\lambda_a'$ is a uniformly bounded real-analytic function for
  small deviations from $\tilde{a}$, we get that $|h'(a)| \sim
  |a^{m-1}|$ which is much larger than
  \[
  n (a) |h(a)| \sim |h(a)| |\log |h(a)||,
  \]
  see equation~(\ref{funch}). Since also $E_{n(a)} (a)$ is uniformly
  bounded, $h'(a)$ is the dominant term in (\ref{xinprim}) for $k =
  n(a)$ as $a \to \tilde{a}$.

  This means that
  \[
  \frac{1}{C} \lambda_a^n \leq |\xi_n'(a)| \leq C \lambda_a^n
  \]
  for $n \sim |\log |h(a)||$, where $C > 1$ is a constant. So in
  particular, $\xi_n'(a)$ is comparable to $(T_a^n)'(X(a)$ for $n \leq
  C_1 |\log |h(a)||$. Moreover, $|\xi_n'(a)| \geq e^{\gamma n}$ for
  some $\gamma > 0$, for $n \sim C_1 |\log |h(a)||$ and we can use the
  above argument in the proof of Lemma~\ref{lem:Xconst} (see also
  Proposition~4.6 in \cite{MA}) to get
  \begin{align*}
    \xi_{N+k+1}'(a) &= T_a'(\xi_{N+k}(a)) \xi_{N+k}'(a) \biggl(1 +
    \frac{ \partial_a T_a(\xi_{N+k}(a)) } {T_a'(\xi_{N+k}(a))
      \xi_{N+k}'(a)} \biggr) \label{xiprim2} \\ &=
    (T_a^{k+1})'(\xi_N(a)) \xi_N'(a) \prod_{j=0}^k
    \biggl(1+\frac{\partial_a T_a(\xi_{N+j}(a)) }{T_a'(\xi_{N+j}(a))
      \xi_{N+j}'(a) } \biggr).
  \end{align*}
  Since $|\xi_n'(a)|$ and $|(T_a^n)'(X(a))|$ grows exponentially and
  $|\partial_a T_a|$ is bounded where it is defined (there are
  finitely many points of discontinuity for $T_a$) we can make the sum
  \[
  \sum_{j=0}^{\infty} \frac{|\partial_a
    T_a(\xi_{N+j}(a)|}{|T_a'(\xi_{N+j}(a)) ||\xi_N'(a)|}
  \]
  arbitrarily small (by also choosing the deviation sufficiently
  small from $\tilde a$). From this we get the desired result from the
  corresponding equation~(\ref{QN}).
\end{proof}


\begin{thebibliography}{00}

\bibitem{MA-Misse} M. Aspenberg, {\em Rational Misiurewicz maps are
  rare}, Communications in Mathematical Physics, 291 (2009), 645--658.

\bibitem{MA} M. Aspenberg, {\em The Collet--Eckmann condition for
  rational functions on the Riemann sphere}, Mathematische
  Zeitschrift, 273 (2013), 935--980.

\bibitem{BenedicksCarleson} M. Benedicks, L. Carleson, {\em On
    iterations of\/ $1-ax^2$ on $(-1,1)$}, Annals of Mathematics (2)
  122 (1985), no. 1, 1--25.

\bibitem{BenedicksCarleson2} M. Benedicks, L. Carleson, {\em On the
  dynamics of the H\'{e}non map}, Annals of Mathematics (2) 133
  (1991), no. 1, 73--169.

\bibitem{BugeaudLiao} Y. Bugeaud, L. Liao, {\em Uniform Diophantine
  approximation related to $b$-ary and $\beta$-expansions}, Ergodic
  Theory and Dynamical Systems, to appear.

\bibitem{BugeaudWang} Y. Bugeaud, B. Wang, {\em Distribution of full
  cylinders and the Diophantine properties of the orbits in
  $\beta$-expansions}, Journal of Fractal Geometry 1 (2014), no. 2,
  221--241.

\bibitem{CarlesonGamelin} L. Carleson, T. Gamelin, {\em Complex
  Dynamics}, Springer, New York, 1993.

\bibitem{FanSchmelingTroubetzkoy} A.-H. Fan, J. Schmeling and
  S. Troubetzkoy, {\em A multifractal mass transference principle for
    Gibbs measures with applications to dynamical Diophantine
    approximation}, Proceedings of the London Mathematical Society,
  107 (2013), 1173--1219.

\bibitem{GeLu} Y. Ge, F. L\"u, {\em A note on inhomogeneous
  Diophantine approximation in beta-dynamical system}, Bulletin of the
  Australian Mathematical Society 91 (2015), no. 1, 34--40.

\bibitem{Gora} P. G\'{o}ra, {\em Invariant densities for generalized
  $\beta$-maps}, Ergodic Theory and Dynamical Systems 27 (2007),
  1583--1598.

\bibitem{HillVelani} R. Hill, S. Velani, {\em The ergodic theory of
  shrinking targets}, Inventiones Mathematicae 119 (1995), no. 1,
  175--198.

\bibitem{Hofbauer} F. Hofbauer, {\em Local dimension for piecewise
    monotonic maps on the interval}, Ergodic Theory and Dynamical
  Systems 15 (1995), no. 6, 1119--1142.

\bibitem{ItoSadahiro} S. Ito, T. Sadahiro, {\em Beta-expansions with
  negative bases}, Integers 9 82009), 239--259.

\bibitem{Ledrappier} F. Ledrappier, {\em Some properties of absolutely
  continuous invariant measures on an interval}, Ergodic Theory
  Dynamical Systems 1 (1981), no. 1, 77--93.

\bibitem{LiPerssonWangWu} B. Li, T. Persson, B. Wang, J. Wu, {\em
  Diophantine approximation of the orbit of 1 in the dynamical system
  of beta expansions}, Mathematische Zeitshrift, April 2014, Volume
  276, Issue 3--4, 799--827.

\bibitem{LiaoSeuret} L. Liao, S. Seuret, {\em Diophantine
  approximation by orbits of expanding Markov maps}, Ergodic Theory
  and Dynamical Systems, 33 (2013), no. 2, 585--608.

\bibitem{LiaoSteiner} L. Liao, W. Steiner, {\em Dynamical properties
  of the negative beta-transformation}, Ergodic Theory and Dynamical
  Systems 32 (2012), no. 5, 1673--1690.

\bibitem{Liverani} C. Liverani, {\em Decay of correlations for piecewise expanding maps}, Journal of Statistical Physics, 78 (1995), no. 3/4, 1111--1129.

\bibitem{LuWu} F. L\"u, J. Wu, {\em Diophantine analysis in
    beta-dynamical systems and Hausdorff dimensions}, Advances in
  Mathematics 290 (2016), 919--937.

\bibitem{Persson} T. Persson, {\em Typical points and families of
  expanding interval mappings},\\ arXiv:1505.07211.

\bibitem{PerssonRams} T. Persson, M. Rams, {\em On Shrinking Targets
  for Piecewise Expanding Interval Maps}, to appear in Ergodic Theory
  and Dynamical Systems, arXiv:1406.6785.

\bibitem{PerssonSchmeling} T. Persson, J. Schmeling, {\em Dyadic
  Diophantine Approximation and Katok's Horseshoe Approximation}, Acta
  Arithmetica 132 (2008), 205--230.

\bibitem{Schnellmann} D. Schnellmann, {\em Typical points for
  one-parameter families of piecewise expanding maps of the interval},
  Discrete and Continuous Dynamical Systems 31 (2011), no. 3,
  877--911.

\bibitem{Wong} S. Wong, {\em Some metric properties of piecewise
  monotonic mappings of the unit interval}, Transactions of the
  American Mathematical Society 246 (1978), 493--500.


\end{thebibliography}
\end{document}